\documentclass[twoside, 10pt]{article}
\usepackage{amsmath,amssymb,amsthm,mathrsfs}
\usepackage[margin=2.0cm]{geometry} 
\usepackage{titlesec,hyperref}
\usepackage{fancyhdr}
\usepackage[numbers,sort&compress]{natbib}
\usepackage{color}
\usepackage[toc,page]{appendix}

\fancypagestyle{plain}
{
\fancyhf{}

}

\titleformat{\subsection}{\it}{\thesubsection.\enspace}{1.5pt}{}
\titleformat{\subsubsection}{\it}{\thesubsubsection.\enspace}{1.5pt}{}

\newtheorem{theo}{Theorem}[section]
\newtheorem{lemm}[theo]{Lemma}

\newtheorem{rema}{Remark}[section]
\numberwithin{equation}{section}

\allowdisplaybreaks

\def\lm{\le}
\def\p{\partial}

\def\th2{\frac{\theta}{2}}
\def\dive{\mathop{\rm div}\nolimits}
\def\r{\varrho}

\def\L{\Lambda}
\def\nd{n_\delta}
\def\rd{\varrho_\delta}
\def\md{m_\delta}

\begin{document}
\title{The Optimal Decay Rate of Strong Solution for the Compressible Nematic Liquid Crystal Equations with Large Initial Data  \hspace{-4mm}}
\author{Jincheng Gao$^\dag$ \quad Zhengzhen Wei$^\ddag$ \quad Zheng-an Yao$^\sharp$ \\[10pt]
\small {School of Mathematics, Sun Yat-Sen University,}\\
\small {510275, Guangzhou, P. R. China}\\[5pt]
}

\footnotetext{Email: \it $^\dag$gaojc1998@163.com,
\it $^\ddag$weizhzh5@mail2.sysu.edu.cn,
\it $^\sharp$mcsyao@mail.sysu.edu.cn}
\date{}

\maketitle

\begin{abstract}
This paper is devoted to establishing the optimal decay rate of the global large solution to compressible nematic liquid crystal equations when the initial perturbation is large and  belongs to $L^1(\mathbb R^3)\cap H^2(\mathbb R^3)$. More precisely, we show that the first and second order spatial derivatives of large solution $(\rho-1, u, \nabla d)(t)$ converges to zero at the $L^2-$rate $(1+t)^{-\frac54}$ and $L^2-$rate $(1+t)^{-\frac74}$ respectively, which are optimal in the sense that they coincide with the decay rates of solution to the heat equation. Thus, we establish optimal decay rate for the second order derivative of global large solution studied in \cite{gao-wei-yao, he-huang-wang} since the compressible nematic liquid crystal flow becomes the compressible Navier-Stokes equations when the director is a constant vector. It is worth noticing that there is no decay loss for the highest-order spatial derivative of solution although the associated initial perturbation is large. Moreover, we also establish the lower bound of decay rates  of $(\rho-1, u, \nabla d)(t)$ itself and its spatial derivative, which coincide with the upper one. Therefore, the decay rates of global large solution $\nabla^2(\rho-1,u,\nabla d)(t)$ $(k=0,1,2)$ are actually optimal.
\vspace*{5pt}

\vspace*{5pt}
\end{abstract}


\section{Introduction}
\quad
In this paper, we are concerned with the upper and lower bounds of decay rates
for a class of global large solution to the three dimensional
compressible nematic liquid crystal equations:
\begin{equation}\label{liqu}
  	\left\{\begin{aligned}
  &\partial_t\rho+\dive(\rho u)=0,\\
   &\partial_t(\rho u)+\dive(\rho u\otimes u)+\nabla P=\dive T-\nabla d\cdot\Delta d,\\
   &\partial_td+u\cdot\nabla d=\Delta d+|\nabla d|^2d,
  	\end{aligned}\right.
  \end{equation}
where $(t, x)\in \mathbb{R}^+\times\mathbb{R}^3$.
The unknown functions $\rho, u=(u_1, u_2, u_3)$ and $P$ represent
the density, velocity and pressure respectively. $d(t,x)\in S^2$, the unit sphere in $\mathbb R^3$, represents the macroscopic average of the nematic liquid crystal orientation field. The pressure $P$  is given by a smooth function $P=P(\rho)=\rho^\gamma$ with the adiabatic exponent $\gamma \ge 1$. And $T$ is the stress tensor given by $T=\mu(\nabla u+\nabla^T u)+\lambda(\dive u) \mathbb I_{3\times3}$ with $ \mathbb I_{3\times3}$ the identity matrix. The constants $\mu$ and $\lambda$ are the viscosity coefficients, which satisfy the following conditions:
$\mu>0$, $2\mu+3\lambda \ge0$. To complete system \eqref{liqu}, the initial data is given by
\begin{equation*}
\left.(\rho, u, d)(t, x)\right|_{t=0}=(\rho_0(x), u_0(x), d_0(x)).
\end{equation*}
 As the space variable tends to infinity, there follows
\begin{equation*}
 \underset{|x|\rightarrow\infty}{\lim}(\rho-1, u, d-\underline d)(t, x)=(0,0,0),
\end{equation*}
where $\underline d$ is a unit constant vector in $S^2$. The systems \eqref{liqu}
are a coupling between the compressible Navier-Stokes equations and a heat flow, which is a macroscopic continuum description of the development for the liquid crystal of nematic type.
In the sequence, we will describe some mathematical results related to the
Navier-Stokes and nematic liquid crystal equations.

\textbf{(I)Some results for the incompressible nematic liquid crystal equations.} The hydrodynamic theory of incompressible liquid crystals was first derived by Ericksen and Leslie in the 1960s (see \cite{erick, les}). It simplified to the incompressible nematic liquid crystal equations, which has been successfully studied. For examples, when density is a constant, Lin et al. \cite{lin-lin-wang} obtained the global existence of the weak solutions in any smooth bounded domain in two dimensions. Gong et al. \cite{gong-huang} obtained a strong global solution if the initial orientation field vector lying in a two-dimensional plane. 
The uniqueness of Leray-Hopf type global weak solution was proved by Lin and Wang \cite{linf-wangc}. Later, Li, Titi and Xin \cite{li-ti-xin} extended their results to the general Ericksen-Leslie system. For more related results in $\mathbb R^2$, one can refer to \cite{hong, xu-zhang,hong-xin,lei, lijinkai, li-jin}. In the case of $\mathbb R^3$, Wang \cite{wang} showed a global well-posedness theory under the condition that $\|u_0\|_{BMO^{-1}}+[d_0]_{BMO}\le \epsilon_0$ for some $\epsilon_0>0$. 
Hineman and Wang \cite{hin-wang} obtained the  local well-posedness in the condition that initial data with small $L^3_{uloc}(\mathbb R^3)-$norm. Lin and Wang \cite{lin-wang} obtained the global existence of a weak solution in the case that the initial director field on the unit upper hemisphere. Recently, Gong et.al \cite{gong-huang-li} constructed infinitely many weak solutions for suitable initial and boundary data. For the results of density-dependent incompressible nematic liquid crystal system, one can refer to \cite{gao-tao, ding-wen, gong-li-xu} and references therein.

\textbf{(II)Some results for the compressible nematic liquid crystal equations.} Let us introduce some related mathematical results. In one-dimensional space, Ding et al. \cite{ding1} obtained both global existence and uniqueness of classical solution of  \eqref{liqu} with H{\"o}lder continuous initial data and non-negative initial density.
This result was generalized to the case of fluid with vacuum in \cite{ding2}.
In dimension three, Jiang et al. \cite{jiang-jiang-wang} obtained the
  global existence of weak solution with large initial energy and without any smallness condition on the initial density and velocity in a bounded domain $\Omega\in \mathbb R^N$, ($N = 2, 3$).
The local-in-time well-posedness of strong solution and some blow-up criterions of breakdown of strong solution were studied in \cite{huang-wang-wen1,huang-wang-wen2}.  The local-in-time strong solutions in $\mathbb R^3$ which under stricter regularity assumptions turn out to be classical was obtained in \cite{ma}.
The global existence of classical solution with smooth initial data which are of small energy but possibly large oscillations in $\mathbb R^3$ were established in \cite{li-xu-zhang}.
As a byproduct, they also studied the large-time behavior of the solution.
Recently, Gao et al.\cite{gao-tao-yao} obtained the global well-posedness
of classical solution under the condition of small perturbation of constant equilibrium state
in the $H^N(\mathbb R^3)(N\ge 3)$-framework.
Furthermore, the optimal decay rate of $k-$th$(k\le N-1)$ order spatial derivative
of solution was obtained in \cite{gao-tao-yao} if the initial perturbation data belongs to $L^1$ additionally.
For more results on the compressible nematic liquid crystal flows \eqref{liqu}, we refer to \cite{lin-lai-wang, huang-ding, li-xu-zhang} and references therein.

\textbf{(III)Some decay results for the compressible Navier-Stokes equations.}
The compressible nematic liquid crystal flow \eqref{liqu} becomes the compressible Navier-Stokes equations(CNS) when the director is a  constant vector. There are many interesting works about the long-time behavior of solution to CNS.
First of all, Matsumura and Nishida \cite{mat-ni} obtained the decay rate
of global small classical solution converging to some constant equilibrium state
in the three-dimensional whole space.
Later, Ponce \cite{ponce} established the optimal $L^p$($p\ge 2$) decay rate for small initial perturbation in $H^l\cap L^1$ with $l\ge3$. With the help of the study of Green function, the optimal $L^p$ ($1\le p\le\infty$) decay rate in $\mathbb R^n$, $n\ge 2$, were obtained by Hoff, Zumbrun \cite{hoff-zumbrun} and Liu, Wang \cite{liu-wang} when the small initial perturbation bounded in $H^s\cap L^1$ with the integer $s\ge[n/2]+3$. For the compressible Navier-Stokes system with an external potential force,
the authors obtained the optimal decay rate in \cite{duan1,duan2,ukai}. For more results about decay problem for the  Navier-Stokes equations, one can refer to \cite{li-zhang, tan-wang, wang-tan}.
If the initial perturbation belongs to some negative Sobolev space $\dot H^{-s}$
rather than some Lebesgue space $L^p$, Guo and Wang \cite{guo-wang} built the time decay rate for the solution of CNS by using a general energy method.
For the case of compressible fluid, there are many results about lower bound of decay rate for
the solution itself of the compressible Navier-Stokes equations \cite{ka-ko, li-zhang},
compressible viscoelastic flows \cite{hu-wu},
and compressible Navier-Stokes-Poisson equations \cite{li-mat-zhang, zhang-li-zhu}.
Later, Gao et al.\cite{gao-lyu} studied the lower bound of decay rate for
the higher order spatial derivative of solution to the compressible Navier-Stokes
and Hall-MHD equations in three-dimensional whole space.
Recently, Wang and Wen \cite{wang-wen} established the optimal time-decay rate for strong solution of the full compressible Navier-Stokes equations with reaction diffusion when the initial perturbation is small in $H^2$. Moreover, they developed a new estimate to avoid the decay loss for the highest-order spatial derivatives of the solution.

However, most of above decay results for the compressible nematic liquid crystal equations and the compressible Navier-Stokes equations are established under the condition that the initial data is a small perturbation of constant equilibrium state.  Recently, He, Huang and Wang \cite{he-huang-wang} proved global stability of large solution to the compressible Navier-Stokes equations. Specifically, under the assumption that $\sup_{t\in \mathbb R^+}\|\rho(t, \cdot)\|_{C^\alpha}\le  M$ for some $0<\alpha<1$, they established upper decay rate
 \begin{equation}\label{pp0}
 \|(\rho-1)(t)\|_{H^1}+\|u(t)\|_{H^1}\le C(1+t)^{-\frac34(\frac2p-1)}.
 \end{equation}
Here the initial perturbation $(\rho_0-1,u_0)\in L^p(\mathbb R^3)\cap H^2(\mathbb R^3)$ with $p\in[1,2)$. The decay result \eqref{pp0} indicates that the first order spatial derivative of solution converges to zero at the $L^2-$rate $(1+t)^{-\frac34(\frac2p-1)}$, which seems not optimal. Meanwhile, the decay rate \eqref{pp0} does not establish the decay rate for the second order spatial derivative of solution.
Thus, for the the global solution studied in \cite{he-huang-wang},
our recent article \cite{gao-wei-yao} established the following decay estimate
\begin{equation}\label{decaypp}
\|\nabla (\rho-1)(t)\|_{H^1}
+\|\nabla u(t)\|_{H^1}
\lm C(1+t)^{-\frac{3}{4}(\frac2p-1)-\frac{1}{2}}.
\end{equation}
Compared with \eqref{pp0}, our result not only established optimal decay rate for the solution's first order spatial derivative, but also proved the second order spatial derivative of global solution will converge to zero. However, \eqref{decaypp} shows that the second order spatial derivative of solution converges to zero at the $L^2-$rate $(1+t)^{-\frac34(\frac2p-1)-\frac12}$, which still seems not optimal.
Recently, Chen et al.\cite{chen-huang} generalized the results in \cite{he-huang-wang} to compressible nematic liquid crystal equations \eqref{liqu} and shown that $(\rho-1, u,\nabla d)$ converges to zero at the $H^1-$rate $(1+t)^{-\frac34}$ when the initial perturbation bounded in $L^1(\mathbb R^3)\cap H^2(\mathbb R^3)$.
It is worth nothing that the optimal decay rate of first order spatial derivative of solution
converging to zero in $L^2-$norm can be proved to be $(1+t)^{-\frac54}$ just taking the
method in our article \cite{gao-wei-yao}.
\textit{However, the decay rate of second order spatial derivative of solution
for the compressible Navier-Stokes and nematic liquid crystal equations obtained in \cite{{he-huang-wang},{chen-huang}} is still not optimal.}
The essential reason is that there are not enough dissipative estimates for the density
to control the energy since the compressible Navier-Stokes equations
are the hyperbolic-parabolic system.

\textit{The purpose of this paper is to establish the optimal decay rate for the large solution of compressible nematic liquid crystal equations \eqref{liqu} when the initial perturbation is  bounded in $L^1(\mathbb R^3)\cap H^2(\mathbb R^3)$.}
Our first target is to establish the optimal decay rate of large solution of \eqref{liqu}
for its first and second order derivatives converging to zero.
Here the decay rate of solution is called optimal just in the sense that
it coincides with the rate of heat equation.
The second purpose is to  establish the lower bound of decay rates for the global strong solution itself and its first and second spatial derivatives for the compressible nematic liquid crystal equations \eqref{liqu}. These lower bound of decay rates will coincide with the upper one.
Therefore, these decay rates obtained in this article are actually optimal.
As a byproduct, we also obtain the optimal decay rate for the second order spatial
derivative of large solution studied in \cite{he-huang-wang} for the
compressible Navier-Stokes in three dimensional whole space.

Before stating the main results of this paper, we would like to introduce some notation which will be used throughout this paper.

 \textbf{Notation:} In this paper, we use $H^s(s\in\mathbb{R}^3)$ to denote the usual Sobolev space with norm $\|\cdot\|_{H^s}$ and $L^p(\mathbb{R}^3)$ to denote the usual $L^p$ space with norm $\|\cdot\|_{L^p}$. $\widehat{f}=\mathcal F(f)$ represents the usual Fourier transform of the function $f$.
 For the sake of simplicity, we write $\int f dx:=\int _{\mathbb{R}^3} f dx$
 and $\|(A, B)\|_X:=\| A \|_X+\| B\|_X$.
 The constant $C$ denotes the generic positive constant independent of time, and may change from line to line. Let $\Lambda^s$ be the pseudodifferential operator defined by $\Lambda^s=\mathcal F^{-1}(|\xi|^s\hat f)$, for $s\in \mathbb R$. We note that $\nabla=(\partial_{x_1},\partial_{x_2},\partial_{x_3})$ and for a multi-index $\alpha=(\alpha_1,\alpha_2,\alpha_3)$, $\partial_x^\alpha=\partial_{x_1}^{\alpha_1}\partial_{x_2}^{\alpha_2}\partial_{x_3}^{\alpha_3}$.

First of all, we recall the following results obtained in \cite{chen-huang}, which will be used in this paper frequently.
\begin{theo}\label{th1}
Let $\mu>\frac12\lambda$, and $(\rho, u, d)$ be a global and smooth solution of \eqref{liqu} with $0\le \rho\le M$, and initial data $(\rho_0, u_0, d_0)$ verifying that $\rho_0\ge c>0$ and the admissible condition
\begin{equation*}
  	\left\{\begin{aligned}
	&u_t|_{t=0}=-u_0\cdot\nabla u_0+\frac{1}{\rho_0}(\dive T_0-\nabla d_0\cdot\Delta d_0-\nabla \rho_0^\gamma),\\
	&d_t|_{t=0}=\Delta d_0+|\nabla d_0|^2d_0-u_0\cdot\nabla d_0,
    	\end{aligned}\right.
  \end{equation*}
and $\sup_{t\in \mathbb R^+}\|\nabla d(t, \cdot)\|_{L^\infty}+\sup_{t\in \mathbb R^+}\|\rho(t, \cdot)\|_{C^\alpha}\le \mathcal M$ for some $0<\alpha<1$.Then if $(\rho_0-1, u_0, \nabla d_0)\in L^1(\mathbb R^3)\cap H^2(\mathbb R^3)$, then there exists a constant $\underline \rho=\underline \rho(c,M,\mathcal M)>0$ such that for all $t\ge 0$, we have
\begin{equation}\label{lbrho}
\rho(t,x)\ge \underline \rho.
\end{equation}
We have the uniform-in-time bounds for the regularity of the solutions, assuming that $\r=\rho-1$, $n=d-\underline d$,
\begin{equation}\label{uni}
\begin{aligned}
&\|\r\|_{L^\infty(H^2)}^2+\|u\|_{L^\infty(H^2)}^2+\|\nabla n\|_{L^\infty(H^2)}^2+\int_0^\infty\big(\|\nabla \r(\tau)\|_{H^1}^2+\|\nabla u(\tau)\|_{H^2}^2+\|\nabla^2n(\tau)\|_{H^2}^2\big)d\tau\\&
\le C(\underline \rho,\mathcal M,\|(\r_0,u_0,\nabla n_0)\|_{L^1\cap H^2},\|d_0\|_{L^2}).
\end{aligned}
\end{equation}
Moreover, we have the decay estimate for the solution
\begin{equation}\label{decay}
\|\r(t)\|_{H^1}+\|u(t)\|_{H^1}+\|\nabla n (t)\|_{H^1}\le C(\underline \rho, \mathcal M,\|\r_0\|_{L^1\cap H^1},\|(u_0,\nabla n_0)\|_{L^1\cap H^2},\|n_0\|_{L^2})(1+t)^{-\frac34}.
\end{equation}
\end{theo}

Our first result can be stated as follows:
\begin{theo}\label{th2}
Suppose all the conditions in Theorem \ref{th1} hold on, and let $(\rho, u, d)$ be the global solution of \eqref{liqu} in Theorem \ref{th1}. Then, it holds on for $k=0,1,2$
\begin{equation}\label{decay00}
\|\nabla^k (\rho-1)(t)\|_{L^2}
+\|\nabla^k  u(t)\|_{L^2}
+\|\nabla^{k+1}d(t)\|_{L^2}
\le C(1+t)^{-\frac34-\frac k2},
\end{equation}
when $t\ge T_2$. Here $C$ is a constant independent of time, and $T_2$ is a large constant given in Lemma \ref{lemm28}.
\end{theo}
\begin{rema}
Compared with decay rate \eqref{decay}, our decay result \eqref{decay00} not only implies that the second order spatial derivative of solution $(\rho, u, \nabla d)$ converges to zero, but also shows that the decay rates for the
first and second order spatial derivatives of solution are optimal in the sense that they coincide with the decay rates of solution to the heat equation.
Specially, as the vector $d(t, x)$ is a constant vector field,
our result also implies the optimal decay rate of
the second order spatial derivative of some class large solution $(\rho, u)$
(see \cite{{he-huang-wang},{gao-wei-yao}})of compressible Navier-Stokes equation is $(1+t)^{-\frac74}$.
\end{rema}

\begin{rema}
By the Sobolev interpolation inequality, it is shown that
the solution $(\rho, u, \nabla d)$ converges to the constant equilibrium
state $(1, 0, 0)$ at the $L^\infty-$rate $(1+t)^{-\frac{3}{2}}$.
\end{rema}

The secod result can be stated as follows:
\begin{theo}\label{th3}
Suppose all the assumption of Theorem \ref{th1} hold on. Denote $m_0:=\rho_0u_0$, $w_0:=\Lambda n_0$, assume that the Fourier transform $\mathcal F(\rho_0,m_0,w_0)=(\widehat{\r_0},\widehat{m_0},\widehat{w_0})$, satisfies $|\widehat{\r_0}|\ge c_0$, $\widehat{m_0}=0$, $|\widehat{w_0}|\ge c_0$, $0< |\xi|\ll1$ with $c_0>0$ a constant. Then the global solution $(\varrho,u,n)$ obtained in Theorem \ref{th1} has the decay rates for large time $t$
\begin{equation*}
c_1(1+t)^{-\frac34-\frac k2}\le \|\nabla^k u(t)\|_{L^2}
\le c_2(1+t)^{-\frac34-\frac k2}, \quad \text{for} \quad k=0,1,2
\end{equation*}
\begin{equation*}
c_1(1+t)^{-\frac34-\frac k2}\le  \|\nabla^k\r(t)\|_{L^2}(t)
\le c_2(1+t)^{-\frac34-\frac k2}, \quad \text{for} \quad k=0,1,2
\end{equation*}
\begin{equation*}
c_3(1+t)^{-\frac34-\frac k2}\le \|\nabla^{k+1}n(t)\|_{L^2}\le c_4(1+t)^{-\frac34-\frac k2}, \quad \text{for} \quad k=0,1,2.
\end{equation*}
Here $c_i$, $(i=1,2,3,4)$ are constants independent of time.
\end{theo}

\begin{rema}
Theorems \ref{th2} and \ref{th3} shows that the lower bound of decay rates for $\nabla^k(\rho-1,u,\nabla d)(t)$ $(k=0,1,2)$ coincide with the upper one, which means these decay rates are actually optimal.
\end{rema}

Next, we would like to introduce the main idea for the proof of Theorem \ref{th2}. First, we shall establish the second order  spatial derivative of solution for system \eqref{liqu} with large initial data. Since we have known that
$\|(\r,u,\nabla n)(t)\|_{H^1}\le C(1+t)^{-\frac34}$, $t\ge 0$. Then these quantities can be small enough after a long time. Therefore, as the strategy mentioned in \cite{mat-ni} for flow with small initial data, we established the second order energy estimate:
\begin{equation}\label{xx1}
\begin{aligned}
&\frac{d}{dt}\|\nabla^2(\r,u,\nabla n)(t)\|_{L^2}^2+\mu\|\nabla^3u\|_{L^2}^2+(\mu+\lambda)\|\nabla^2\dive u\|_{L^2}^2+\|\nabla^4 n\|_{L^2}^2\\&
\lesssim Q_1(t)(\|\nabla^2\r\|_{L^2}^2+\|\nabla^2u\|_{H^1}^2+\|\nabla^3n\|_{H^1}^2),
\end{aligned}
\end{equation}
where $Q_1(t)$ consists of some difficult terms, such as $\|(\r, n)\|_{L^\infty}$, $\|\nabla (\r,u,\nabla n)\|_{L^3}$. According to Sobolev interpolation inequality, these terms should be controlled by the product of $\|(\r,u,\nabla n)(t)\|_{L^2}$ and $\|\nabla^2(\r, u,\nabla n)\|_{L^2}$. Therefore, $Q_1(t)$ is a small quantity after a long time according to the decay result \eqref{decay} and uniform result \eqref{uni}. In order to get the dissipative estimate for $\nabla^2\r$, we establish the following estimate
\begin{equation}\label{xx2}
\frac{d}{dt}\int\nabla u\cdot\nabla^2\r dx+C\|\nabla^2\r\|_{L^2}^2
\lesssim \|\nabla^2u(t)\|_{H^1}^2+Q_2(t)(\|\nabla^2\r\|_{L^2}^2+\|\nabla^2u\|_{H^1}^2+\|\nabla^3n\|_{H^1}^2),
\end{equation}
where $Q_2(t)$ is similar as $Q_1(t)$. The combination of \eqref{xx1} and \eqref{xx2} implies that
\begin{equation}\label{xxx3}
\begin{aligned}
&\frac{d}{dt}X_h(t)+\frac12(\mu\|\nabla^3u\|_{L^2}+(\mu+\lambda)\|\nabla^2\dive u\|_{L^2}^2+\|\nabla^4n\|_{L^2}^2)+C\delta\|\nabla^2\r\|_{L^2}^2\\&
\le C_2\delta(\|\nabla^2u\|_{L^2}^2+\|\nabla^3n\|_{L^2}^2), \quad \text{for}\quad t\ge T_1,
\end{aligned}
\end{equation}
where
\begin{equation*}
X_h(t):=\frac12\big(\|\nabla^2\r\|_{L^2}^2+\|\nabla^2u\|_{L^2}^2+\|\nabla^3n\|_{L^2}^2+2\delta\int\nabla u\cdot\nabla^2\r dx\big),
\end{equation*}
 $\delta$ is a small constant and $T_1$ is a large constant. In the following, we focus our attention on how to obtain the optimal decay rate of the second order spatial derivative of large solution for system \eqref{liqu}.
Actually, the presence of the term $\int\nabla u\cdot\nabla^2\r dx$ is the reason that the decay rate of  the second order derivative of solution, i.e. $\|\nabla^2(\r,u,\nabla n)\|_{L^2}$ is not optimal. Indeed, $X_h(t)\sim \|(\nabla\r,\nabla u)\|_{H^1}^2+\|\nabla^3n\|_{L^2}^2$ since the smallness of $\delta$. Thus the second order and the first order spatial derivatives of the solution have the same time decay rate. In order to avoid this obstacle, by using a method as in \cite{wang-wen}, we need to establish a new estimate for $\int\nabla u\cdot\nabla^2\r^L dx$, where $\r^L$ stand for the low-medium-frequency part of $\r$ (see the definition in Appendix \ref{appendixB}). Then removing this term from \eqref{xxx3}, using the good properties of the low-frequency and high-frequency decomposition, we deduce that
\begin{equation*}
\frac{d}{dt}\big(X_h(t)-\delta\int\nabla u\cdot\nabla^2\r^Ldx\big)+C_4\big(X_h(t)-\delta\int\nabla u\cdot\nabla^2\r^Ldx\big)\le C\|\nabla^2(\r^L,u^L,\nabla n^L)\|_{L^2}^2, \quad t\ge T_1.
\end{equation*}
By noting that
\begin{equation*}
X_h(t)-\delta\int\nabla u\cdot\nabla^2\r^Ldx=\frac12\|\nabla^2(\r,u,\nabla n)\|_{L^2}^2+\delta\int\nabla u\cdot\nabla^2\r^hdx\sim\|\nabla^2(\r,u,\nabla n)\|_{L^2}^2,
\end{equation*}
it is easy to obtain that
\begin{equation*}
\begin{aligned}
\|\nabla^2(\r,u,\nabla n)(t)\|_{L^2}^2&\le Ce^{-Ct}\|\nabla^2(\r,u,\nabla n)(T_1)\|_{L^2}^2
+C\int_{T_1}^te^{-C(t-\tau)}\|\nabla^2(\r^L,u^L,\nabla n^L)(\tau)\|_{L^2}^2d\tau.
\end{aligned}
\end{equation*}
 The Duhamel's principle allow us get fast enough decay rate of $\|\nabla^2(\r^L,u^L,\nabla n^L)(\tau)\|_{L^2}^2$, which make it possible to prove the optimal decay rate of $\|\nabla^2(\r,u,\nabla n)\|_{L^2}^2$.

The rest of this paper is organized as follows. In Section 2, we will give the proof of Theorem \ref{th2}. In Section 3, we establish the lower bound decay of strong solution to \eqref{liqu} with large initial data, which completes the proof of Theorem \ref{th3}. In Appendix \ref{appendixA}, we give the decay estimates of the low-medium-frequency part for the linearized system. In Appendix \ref{appendixB}, we give the definition of the frequency decomposition and some known inequalities.

\section{Proof of Theorem \ref{th2}}

\quad
In this section, we will give the proof for the Theorem \ref{th2}.
The analysis proceeds in several steps, which we will give now in detail below.

\subsection{Energy estimate}

\quad
Since the solution itself and its first order spatial derivatives admit the same
$L^2-$rate $(1+t)^{-\frac34}$, these quantities can be small enough essentially
if the time is large.
Thus, we will take the strategy of the frame of small initial data(cf. \cite{mat-ni})
to establish the energy estimate.
Denoting $\varrho:=\rho-1$, $n:=d-\underline d$, we rewrite \eqref{liqu} in the perturbation form as follows
\begin{equation}\label{liqu1}
  	\left\{\begin{aligned}
	&\p_t\r+\dive u=S_1,\\
	&\p_t u-\mu\Delta u-(\mu+\lambda)\nabla\dive u+P'(1)\nabla\r=S_2,\\
	&\p_t n-\Delta n=S_3,
    	\end{aligned}\right.
  \end{equation}
  where the nonlinear terms $S_1$, $S_2$ and $S_3$ are defined by
  \begin{equation*}
  	\left\{\begin{aligned}
	&S_1:=-\r\dive u-u\cdot\nabla\r,\\
	&S_2:=-u\cdot\nabla u-h(\r)\big(\mu\Delta u+(\mu+\lambda)\nabla\dive u\big)-f(\r)\nabla\r-g(\r)\nabla d\cdot\Delta d,\\
	&S_3:=-u\cdot\nabla n+|\nabla n|^2(n+\underline d),
    	\end{aligned}\right.
  \end{equation*}
where
\begin{equation*}
h(\r):=\frac{\r}{\r+1}, \quad f(\r):=\frac{P'(\r+1)}{\r+1}-\frac{P'(1)}{1}, \quad g(\r):=\frac{1}{\r+1}.
\end{equation*}

First, we give the first order spatial derivative estimate as follows.

\begin{lemm}
Under the assumptions of Theorem \ref{th1}, the global solution $(\r, u, n)$ of Cauchy problem \eqref{liqu1} has the estimate
\begin{equation*}
\begin{aligned}
&\frac12\frac{d}{dt}\int(|\nabla\r|^2+|\nabla u|^2+|\nabla^2 n|^2)dx+\int(\mu|\nabla^2 u|^2+(\mu+\lambda)|\nabla\dive u|^2+|\nabla^3n|^2)dx\\&
\le C\big(\|\r\|_{H^1}+\|u\|_{H^1}+\|\r\|_{L^2}^{\frac14}+\|\nabla n\|_{H^1}+\|\nabla n\|_{H^1}^2\big)\big(\|\nabla^2u\|_{L^2}^2+\|\nabla^2\r\|_{L^2}^2+\|\nabla^3n\|_{L^2}^2+\|\nabla\dive u\|_{L^2}^2\big).
\end{aligned}
\end{equation*}
\end{lemm}

\begin{proof}
Applying differential operator $\nabla$ to $\eqref{liqu1}_1$, $\eqref{liqu1}_2$ respectively and $\nabla^2$ to $\eqref{liqu1}_3$, then multiplying the resulting identities by $\nabla \r$, $\nabla u$ and $\nabla^2n$ respectively and integrating over $\mathbb R^3$, it is easy to obtain
\begin{equation}\label{a1}
\begin{aligned}
&\frac12\frac{d}{dt}\int(|\nabla\r|^2+|\nabla u|^2+|\nabla^2n|^2)dx+\int(\mu|\nabla^2u|^2+(\mu+\lambda)|\nabla\dive u|^2+|\nabla^3n|^2)dx\\&
=\int \nabla S_1\cdot\nabla\r dx+\int\nabla S_2\cdot\nabla udx+\int \nabla^2 S_3\cdot\nabla^2ndx.
\end{aligned}
\end{equation}
Integrating by part and applying the H\"{o}lder inequality yield
\begin{equation}\label{a2}
\big|\int \nabla S_1\cdot\nabla\r dx\big|\le \|S_1\|_{L^2}\|\nabla^2\r\|_{L^2}.
\end{equation}
Using the H\"{o}lder and Sobolev inequalities, we show that
\begin{equation}\label{d1}
\|S_1\|_{L^2}
\le \|\r\|_{L^3}\|\dive u\|_{L^6}+\|u\|_{L^3}\|\nabla\r\|_{L^6}
\le C(\|\r\|_{H^1}+\|u\|_{H^1})(\|\nabla^2 u\|_{L^2}+\|\nabla^2\r\|_{L^2}).
\end{equation}
This and inequality \eqref{a2} give
\begin{equation}\label{a22}
\big|\int \nabla S_1\cdot\nabla\r dx\big|\le C(\|\r\|_{H^1}+\|u\|_{H^1})(\|\nabla^2 u\|^2_{L^2}+\|\nabla^2\r\|^2_{L^2}).
\end{equation}
Applying the H\"{o}lder and Sobolev inequalities, we get
\begin{equation}\label{a3}
\|u\cdot\nabla u\|_{L^2}\le \|u\|_{L^3}\|\nabla u\|_{L^6}\le C\|u\|_{H^1}\|\nabla^2u\|_{L^2}.
\end{equation}
Using the lower bound of density \eqref{lbrho}, Sobolev inequality and uniform estimate \eqref{uni}, we have
\begin{equation}\label{a4}
\begin{aligned}
\|h(\r)(\mu\Delta u+(\mu+\lambda)\nabla\dive u)\|_{L^2}
\le C\|\r\|_{L^\infty}\|\nabla^2u\|_{L^2}
\le C\|\r\|_{L^2}^{\frac14}\|\nabla^2\r\|_{L^2}^{\frac34}\|\nabla^2 u\|_{L^2}
\le C\|\r\|_{L^2}^{\frac14}\|\nabla^2u\|_{L^2}.
\end{aligned}
\end{equation}
Using the Taylor expression, H\"{o}lder and Sobolev inequalities, we get
\begin{equation}\label{a5}
\|h(\r)\nabla\r\|_{L^2}\le C\|\r\|_{L^3}\|\nabla\r\|_{L^6}\le C\|\r\|_{H^1}\|\nabla^2\r\|_{L^2}.
\end{equation}
By the lower bound of density \eqref{lbrho} and the Sobolev inequality, we obtain that
\begin{equation}\label{a6}
\|g(\r)\nabla n\Delta n\|_{L^2}\le C\|\nabla n\|_{L^3}\|\Delta n\|_{L^6}\le C\|\nabla n\|_{H^1}\|\nabla^3n\|_{L^2}.
\end{equation}
The combination of \eqref{a3}, \eqref{a4}, \eqref{a5} and \eqref{a6} gives
\begin{equation}\label{a10}
\big|\int\nabla S_2\cdot\nabla udx\big|
\le \|S_2\|_{L^2}\|\nabla^2u\|_{L^2}
\le (\|u\|_{H^1}+\|\r\|_{L^2}^{\frac14}+\|\r\|_{H^1}+\|\nabla n\|_{H^1})
(\|\nabla^2u\|_{L^2}^2+\|\nabla^2\r\|_{L^2}^2+\|\nabla^3n\|_{L^2}^2).
\end{equation}
The routine calculation yields directly
\begin{equation*}
\nabla S_3=-\nabla u\cdot\nabla n-u\cdot\nabla^2n+2|\nabla n|\cdot\nabla^2n\cdot(n+\underline d)+|\nabla n|^2\nabla n.
\end{equation*}
It is easy to see that
\begin{equation}\label{a7}
\begin{aligned}
\|\nabla u\cdot\nabla n\|_{L^2}
\!+\!\|u\cdot\nabla^2n\|_{L^2}
&\le \|\nabla n\|_{L^3}\|\nabla u\|_{L^6}
     \!+\!\|u\|_{L^3}\|\nabla^2n\|_{L^6}
\le C(\|\nabla n\|_{H^1}\|\nabla^2u\|_{L^2}
       \!+\!\|u\|_{H^1}\|\nabla^3n\|_{L^2}).
\end{aligned}
\end{equation}
Using the uniform bound \eqref{uni}, we have
\begin{equation}\label{a8}
\begin{aligned}
\|2|\nabla n|\cdot\nabla^2n\cdot(n+\underline d)\|_{L^2}&\le C(\|n\|_{L^\infty}+\underline d)\|\nabla n\|_{L^3}\|\nabla^2n\|_{L^6}\\&
\le C(\|\nabla n\|_{L^2}^{\frac12}\|\nabla^2n\|_{L^2}^{\frac12}+\underline d)\|\nabla n\|_{H^1}\|\nabla^3n\|_{L^2}\\&
\le C\|\nabla n\|_{H^1}\|\nabla^3n\|_{L^2}.
\end{aligned}
\end{equation}
The Sobolev inequality gives
\begin{equation}\label{a9}
\begin{aligned}
\||\nabla n|^2\cdot\nabla n\|_{L^2}
\le 2\|\nabla n\|_{L^3}\|\nabla n\cdot\nabla^2n\|_{L^2}
\le C\|\nabla n\|_{L^3}^2\|\nabla^2n\|_{L^6}
\le C \|\nabla n\|_{H^1}^2\|\nabla^3n\|_{L^2}.
\end{aligned}
\end{equation}
The combination of \eqref{a7}, \eqref{a8} and \eqref{a9} implies
\begin{equation}\label{a11}
\big|\int \nabla^2 S_3\cdot\nabla^2ndx\big|\le C(\|\nabla n\|_{H^1}+\|u\|_{H^1}+\|\nabla n\|_{H^1}^2)(\|\nabla^2u\|_{L^2}^2+\|\nabla^3n\|_{L^2}^2).
\end{equation}
Therefore, combine the estimates \eqref{a1}, \eqref{a22}, \eqref{a10} and \eqref{a11}, we completes the proof of this lemma.
\end{proof}

Next, we establish the energy estimate for the second order spatial derivative of solution $(\r,u,n)$ for the Cauchy problem \eqref{liqu1}, which can help us to achieve the decay rate for them.
\begin{lemm}\label{2un}
Under the assumptions of Theorem \ref{th1}, the global solution $(\r,u,n)$ of Cauchy problem \eqref{liqu1} has the estimate
\begin{equation}\label{p1}
\begin{aligned}
&\frac12\frac{d}{dt}\int(|\nabla^2\r|^2+|\nabla ^2u|^2+|\nabla^3n|^2)dx+\int(\mu|\nabla^3 u|^2+(\mu+\lambda)|\nabla^2\dive u|^2+|\nabla^4n|^2)dx\\&
\le C\big(\|(\r, u, \nabla u,\nabla n)\|_{L^2}^{\frac14}+\|(u,\nabla n)\|_{H^1}+\|\nabla^2n\|_{L^2}^{\frac12}\big)\big(\|\nabla^2\r\|_{L^2}^2+\|\nabla^2 u\|_{H^1}^2+\|\nabla^3n\|_{H^1}^2\big).
\end{aligned}
\end{equation}
\end{lemm}
\begin{proof}
Applying $\nabla^2$ to $\eqref{liqu1}_1$, $\eqref{liqu1}_2$ respectively and $\nabla^3$ to $\eqref{liqu1}_3$, then multiplying the resulting identities by $\nabla^2 \r$, $\nabla^2 u$ and $\nabla^3n$ respectively and integrating over $\mathbb R^3$, it is easy to obtain
\begin{equation}\label{b0}
\begin{aligned}
&\frac12\frac{d}{dt}\int(|\nabla^2\r|^2+|\nabla^2 u|^2+|\nabla^3n|^2)dx+\int(\mu|\nabla^3u|^2+(\mu+\lambda)|\nabla^2\dive u|^2+|\nabla^4n|^2)dx\\&
=\int \nabla^2 S_1\cdot\nabla^2\r dx+\int\nabla^2 S_2\cdot\nabla^2 udx+\int \nabla^3S_3\cdot\nabla^3ndx.
\end{aligned}
\end{equation}
Recall that $S_1:=-\r\dive u-u\cdot\nabla\r$, the direct computation gives
\begin{equation*}
\nabla^2(\r\dive u)=\r\nabla^2\dive u+2\nabla\r\cdot\nabla\dive u+\nabla^2\r\cdot\dive u,
\end{equation*}
then it holds
\begin{equation}\label{b1}
\begin{aligned}
&\|\nabla^2(\r\dive u)\|_{L^2}\|\nabla^2\r\|_{L^2}\\&
\le (\|\r\|_{L^\infty}\|\nabla^2\dive u\|_{L^2}+\|\nabla\r\|_{L^3}\|\nabla\dive u\|_{L^6})\|\nabla^2\r\|_{L^2}+\|\dive u\|_{L^\infty}\|\nabla^2\r\|_{L^2}^2\\&
\le C(\|\r\|_{L^\infty}+\|\nabla\r\|_{L^3})(\|\nabla^3u\|_{L^2}^2+\|\nabla^2\r\|_{L^2}^2)+C\|\dive u\|_{L^\infty}\|\nabla^2\r\|_{L^2}^2.
\end{aligned}
\end{equation}
By routine checking, one may check that
\begin{equation*}
\nabla^2(u\cdot\nabla\r)=u\cdot\nabla(\nabla^2\r)+2\nabla u\cdot\nabla^2\r+\nabla^2u\cdot\nabla\r.
\end{equation*}
The integration by part yields directly
\begin{equation*}
\int u\cdot\nabla(\nabla^2\r)\cdot\nabla^2\r dx=\frac12\int u\cdot\nabla(|\nabla^2\r|^2)dx=-\frac12\int\dive u|\nabla^2\r|^2dx,
\end{equation*}
and hence, we have
\begin{equation}\label{b2}
\begin{aligned}
&\big|\int\nabla^2(u\cdot\nabla\r)\cdot\nabla^2\r dx\big|\\
&\le\|\dive u\|_{L^\infty}\|\nabla^2\r\|_{L^2}^2+(\|\nabla u\|_{L^\infty}\|\nabla^2\r\|_{L^2}+\|\nabla\r\|_{L^3}\|\nabla^2u\|_{L^6})\|\nabla^2\r\|_{L^2}\\&
\le \|\nabla u\|_{L^\infty}\|\nabla^2\r\|_{L^2}^2+C\|\nabla\r\|_{L^3}(\|\nabla^3u\|_{L^2}^2+\|\nabla^2\r\|_{L^2}^2).
\end{aligned}
\end{equation}
The combination of \eqref{b1} and \eqref{b2} gives
\begin{equation}\label{b3}
\begin{aligned}
\big|\int\nabla^2S_1\cdot\nabla^2\r dx\big|\le C(\|\nabla\r\|_{L^3}+\|\r\|_{L^\infty})(\|\nabla^2\r\|_{L^2}^2+\|\nabla^3u\|_{L^2}^2)+C\|\nabla u\|_{L^\infty}\|\nabla^2\r\|_{L^2}^2.
\end{aligned}
\end{equation}
By virtue of the Sobolev inequality and the uniform estimate \eqref{uni}, it follows that
\begin{equation}\label{b4}
\begin{aligned}
\|\r\|_{L^\infty}+\|\nabla\r\|_{L^3}+\|\nabla u\|_{L^3}
&\le C(\|\r\|_{L^2}^{\frac14}\|\nabla^2\r\|_{L^2}^{\frac34}+\|\r\|_{L^2}^{\frac14}\|\nabla^2\r\|_{L^2}^{\frac34}+\|u\|_{L^2}^{\frac14}\|\nabla^2u\|_{L^2}^{\frac34})\\&
\le C(\|\r\|_{L^2}^{\frac14}+\|u\|_{L^2}^{\frac14}),
\end{aligned}
\end{equation}
and
\begin{equation}\label{b5}
\begin{aligned}
\|\nabla u\|_{L^\infty}\|\nabla^2\r\|_{L^2}^2&\le C \|\nabla u\|_{L^2}^{\frac14}\|\nabla^3u\|_{L^2}^{\frac34}\|\nabla^2\r\|_{L^2}^{\frac54}\|\nabla^2\r\|_{L^2}^{\frac34}\\&
\le C\|\nabla u\|_{L^2}^{\frac14}\|\nabla^3u\|_{L^2}^{\frac34}\|\nabla^2\r\|_{L^2}^{\frac54}\\&
\le C\|\nabla u\|_{L^2}^{\frac14}(\|\nabla^3u\|_{L^2}^2+\|\nabla ^2\r\|_{L^2}^2).
\end{aligned}
\end{equation}
Substituting \eqref{b4} and \eqref{b5} into \eqref{b3}, we get
\begin{equation}\label{c1}
\big|\int\nabla^2S_1\cdot\nabla^2\r dx\big|\le C(\|\r\|_{L^2}^{\frac14}+\|\nabla u\|_{L^2}^{\frac14})(\|\nabla^2\r\|_{L^2}^2+\|\nabla^3u\|_{L^2}^2).
\end{equation}
Next, we shall estimate the second term on the right hand side of \eqref{b0}. Integrating by part, we get
\begin{equation}\label{b10}
\int\nabla^2 S_2\cdot\nabla^2udx=-\int\nabla S_2\cdot\nabla^3udx.
\end{equation}
Straightforward calculation shows that
   \begin{equation}\label{b6}
   \begin{split}
   \nabla S_2=
   &-\nabla u\cdot \nabla u-u\cdot \nabla(\nabla u)
     -\frac{\r}{1+\r}[\mu\nabla\Delta u+(\mu+\lambda)\nabla^2\dive u]\\
   & -\Big[\frac{P'(1+\r)}{1+\r}-\frac{P'(1)}{1}\Big]\nabla^2\r
     -\frac{\nabla\r}{(1+\r)^2}[\mu\Delta u+(\mu+\lambda)\nabla\dive u]\\
   &-\frac{P''(1+\r)(1+\r)-P'(1+\r)}{(1+\r)^2}\nabla\r\nabla\r-\frac{1}{\r+1}\nabla(\nabla n)\Delta n\\&
   -\frac{1}{\r+1}\nabla n\cdot\nabla\Delta n
   +\frac{\nabla\r}{(1+\r)^2}\nabla n\cdot\Delta n.
   \end{split}
   \end{equation}
Observe that
\begin{equation*}
P'(1+\r)=\gamma(1+\r)^{\gamma-1}, \qquad P''(1+\r)=\gamma(\gamma-1)(1+\r)^{\gamma-2},
\end{equation*}
and hence, it holds on
\begin{equation}\label{b7}
\frac{P''(1+\r)(1+\r)-P'(1+\r)}{(1+\r)^2}=\gamma(\gamma-2)(1+\r)^{\gamma-3}, \qquad \frac{P'(1+\r)}{1+\r}=\gamma(1+\r)^{\gamma-2}.
\end{equation}
The combination of \eqref{b6} and \eqref{b7} yields directly
\begin{equation}\label{b8}
\begin{aligned}
\|\nabla S_2 \|_{L^2}
&\lm C(\|\nabla u\|_{L^3}\|\nabla u\|_{L^6}
     +\|u\|_{L^3}\|\nabla^2 u\|_{L^6}
     +\|\r\|_{L^{\infty}}\|\nabla^3 u\|_{L^2})\\&
     \quad +C(\|\r\|_{L^{\infty}}\|\nabla^2\r\|_{L^2}
     +\|\nabla\r\|_{L^3}\|\nabla^2u\|_{L^6}+\|\nabla\r\|_{L^3}\|\nabla\r\|_{L^6})\\&
    \quad +C(\|\nabla^2n\|_{L^3}\|\nabla^2n\|_{L^6}+\|\nabla n\|_{L^3}\|\nabla^3n\|_{L^6}+\|\nabla n\|_{L^\infty}\|\nabla\r\|_{L^3}\|\nabla^2n\|_{L^6})\\&
    \le C(\|\nabla u\|_{L^3}+\|u\|_{H^1}+\|\r\|_{L^\infty}+\|\nabla \r\|_{L^3})(\|\nabla^2u\|_{H^1}+\|\nabla^2\r\|_{L^2})\\&
    \quad +(\|\nabla^2n\|_{L^3}+\|\nabla n\|_{L^\infty}\|\nabla\r\|_{L^3}+\|\nabla n\|_{H^1})\|\nabla^3n\|_{H^1}.
\end{aligned}
\end{equation}
By virtue of the Sobolev inequality and the uniform estimate \eqref{uni}, it follows that
\begin{equation}\label{b9}
\begin{aligned}
\|\nabla^2n\|_{L^3}+\|\nabla n\|_{L^\infty}\|\nabla\r\|_{L^3}&
\le C(\|\nabla n\|_{L^2}^{\frac14}\|\nabla^3n\|_{L^2}^{\frac34}+\|\nabla n\|_{H^2}\|\r\|_{L^2}^{\frac14}\|\nabla^2\r\|_{L^2}^{\frac34})\\&
\le C(\|\nabla n\|_{L^2}^{\frac14}+\|\r\|_{L^2}^{\frac14}).
\end{aligned}
\end{equation}
Substituting \eqref{b4} and \eqref{b9} into \eqref{b8}, we obtain that
\begin{equation}\label{d2}
\|\nabla S_2\|_{L^2}\le C(\|u\|_{L^2}^{\frac14}+\|u\|_{H^1}+\|\r\|_{L^2}^{\frac14}+\|\nabla n\|_{L^2}^{\frac14}+\|\nabla n\|_{H^1})(\|\nabla^2u\|_{H^1}+\|\nabla^2\r\|_{L^2}+\|\nabla^3n\|_{H^1}).
\end{equation}
This together with \eqref{b10} implies that
\begin{equation}\label{c2}
\begin{aligned}
\big|\int\nabla^2 S_2\cdot\nabla^2udx\big|
\le C(\|u\|_{L^2}^{\frac14}+\|u\|_{H^1}+\|\r\|_{L^2}^{\frac14}
      +\|\nabla n\|_{L^2}^{\frac14}+\|\nabla n\|_{H^1})(\|\nabla^2u\|^2_{H^1}+\|\nabla^2\r\|^2_{L^2}+\|\nabla^3n\|^2_{H^1}).
\end{aligned}
\end{equation}
Finally, we focus on the last term on the right hand side of \eqref{b0}.
Regularly computation shows that
\begin{equation*}
\nabla^2S_3
=-\nabla^2u\cdot\nabla n-2\nabla u\cdot\nabla^2n-u\cdot\nabla^3n
+2|\nabla^2n|^2(n+\underline d)
+2|\nabla n||\nabla^3n|(n+\underline d)+3|\nabla^2n||\nabla n|^2.
\end{equation*}
Using the H\"{o}lder and Sobolev inequalities, we have
\begin{equation}\label{b11}
\begin{aligned}
\|\nabla^2 S_3\|_{L^2}&\le \|\nabla n\|_{L^3}\|\nabla^2 u\|_{L^6}+\|\nabla u\|_{L^3}\|\nabla^2n\|_{L^6}+\|u\|_{L^\infty}\|\nabla^3n\|_{L^2}\\&
\quad +(\|n\|_{L^\infty}+\underline d)(\|\nabla^2n\|_{L^3}\|\nabla ^2n\|_{L^6}+\|\nabla n\|_{L^3}\|\nabla^3n\|_{L^6})\\&
\quad +\|\nabla n\|_{L^\infty}\|\nabla n\|_{L^3}\|\nabla^2n\|_{L^6}\\&
\le C(\|\nabla n\|_{L^3}+\|\nabla u\|_{L^3}+\|u\|_{L^\infty})(\|\nabla^3n\|_{L^2}+\|\nabla^3u\|_{L^2})\\&
\quad +C(\|n\|_{L^\infty}+\underline d)(\|\nabla^2n\|_{L^3}\|\nabla ^3n\|_{L^2}+\|\nabla n\|_{L^3}\|\nabla^4n\|_{L^2})\\&
\quad +C\|\nabla n\|_{L^\infty}\|\nabla n\|_{L^3}\|\nabla^3n\|_{L^2}.
\end{aligned}
\end{equation}
By the Sobolev inequality and the uniform estimate \eqref{uni},
it is easy to check that
\begin{equation}\label{b12}
\|n\|_{L^\infty}+\|\nabla n\|_{L^3}\le C \|\nabla n\|_{L^2}^{\frac12}\|\nabla^2n\|_{L^2}^{\frac12}\le C\|\nabla^2n\|_{L^2}^{\frac12},
\end{equation}
and
\begin{equation}\label{b13}
\|\nabla^2n\|_{L^3}\le C\|\nabla^2n\|_{L^2}^{\frac12}\|\nabla^3n\|_{L^2}^{\frac12}\le C\|\nabla^2n\|_{L^2}^{\frac12}.
\end{equation}
Substituting \eqref{b4}, \eqref{b12} and \eqref{b13} into \eqref{b11}, we get
\begin{equation*}
\|\nabla^2S_3\|_{L^2}\le C(\|\nabla^2n\|_{L^2}^{\frac12}+\|u\|_{L^2}^{\frac14}+\|\nabla^2n\|_{L^2})(\|\nabla^3n\|_{H^1}+\|\nabla^3u\|_{L^2}),
\end{equation*}
which implies that
\begin{equation}\label{c3}
\begin{aligned}
\big|\int \nabla^3S_3\cdot\nabla^3ndx\big|
\le \|\nabla^2S_3\|_{L^2}\|\nabla^4n\|_{L^2}
\le C (\|\nabla^2n\|_{L^2}^{\frac12}+\|u\|_{L^2}^{\frac14}+\|\nabla^2n\|_{L^2})
       (\|\nabla^3n\|^2_{H^1}+\|\nabla^3u\|^2_{L^2}).
\end{aligned}
\end{equation}
Then the combination of \eqref{c1}, \eqref{c2} and \eqref{c3} completes the proof.
\end{proof}

In order to close the energy estimate, it is necessary to establish the dissipation estimate for $\nabla^2\r$.

\begin{lemm}\label{2rho}
Under the assumptions of Theorem \ref{th1}, the global solution $(\r,u,n)$ of Cauchy problem \eqref{liqu1} has the estimate
\begin{equation}\label{p2}
\begin{aligned}
&\frac{d}{dt}\int\nabla u\cdot\nabla^2\r dx+\frac{3P'(1)}{4}\int |\nabla^2\r|^2dx\\&
\le C\|\nabla^2u\|_{H^1}^2+(\|(\r,u,\nabla n)\|_{H^1}+\|(\r,u,\nabla n)\|_{L^2}^{\frac14})(\|\nabla^2\r\|_{L^2}^2+\|\nabla^2u\|_{H^1}^2+\|\nabla^3n\|_{H^1}^2).
\end{aligned}
\end{equation}
\end{lemm}
\begin{proof}
Applying $\nabla$ operator to $\eqref{liqu1}_2$ and multiplying the resulting by $\nabla^2\r$, and integrating over $\mathbb R^3$, we have
\begin{equation*}
\begin{aligned}
\int\p_t(\nabla u)\cdot\nabla^2\r dx-\int(\mu\Delta\nabla u+(\mu+\lambda)\nabla^2\dive u)\cdot\nabla^2\r dx+P'(1)\int |\nabla^2\r|^2dx=\int\nabla S_2\cdot\nabla^2\r dx.
\end{aligned}
\end{equation*}
Using the first equation in \eqref{liqu1}, it holds on
\begin{equation*}
\begin{aligned}
\int\p_t(\nabla u)\cdot\nabla^2\r dx&=\frac{d}{dt}\int\nabla u\cdot\nabla^2\r dx-\int\nabla u\cdot\nabla^2\r_tdx\\&
=\frac{d}{dt}\int\nabla u\cdot\nabla^2\r dx+\int\nabla\dive u\cdot\nabla\r_t dx\\&
=\frac{d}{dt}\int\nabla u\cdot\nabla^2\r dx+\int\nabla\dive u\cdot(\nabla S_1-\nabla\dive u) dx.
\end{aligned}
\end{equation*}
Then integrating by part, and using the  H\"{o}lder and Cauchy inequalities, we obtain that
\begin{equation*}
\begin{aligned}
&\frac{d}{dt}\int\nabla u\cdot\nabla^2\r dx+P'(1)\int |\nabla^2\r|^2dx\\&=\int(\mu\Delta\nabla u+(\mu+\lambda)\nabla^2\dive u)\cdot \nabla^2\r dx+\int\nabla S_2\cdot\nabla^2\r dx+\int\nabla\dive u\cdot(\nabla\dive u-\nabla S_1)dx\\&
\le \frac14P'(1)\|\nabla^2\r\|_{L^2}^2+C\|\nabla^2 u\|_{H^1}^2+C(\|\nabla S_2\|_{L^2}^2+\|S_1\|_{L^2}^2),
\end{aligned}
\end{equation*}
which, together with \eqref{d1} and \eqref{d2} completes the proof of this lemma.
\end{proof}

Combining the estimates obtained in Lemmas \ref{2un} and \ref{2rho},
we derive the following energy estimate.

\begin{lemm}\label{t1}
Under the assumptions of Theorem \ref{th1}, we define
\begin{equation*}
X_h(t):=\frac12\big(\|\nabla^2\r\|_{L^2}^2+\|\nabla^2u\|_{L^2}^2+\|\nabla^3n\|_{L^2}^2+2\delta\int\nabla u\cdot\nabla^2\r dx\big).
\end{equation*}
Then there exists a large enough time $T_1>0$, such that
\begin{equation}\label{e1}
\begin{aligned}
&\frac{d}{dt}X_h(t)+\frac12\int(\mu|\nabla^3u|^2+(\mu+\lambda)|\nabla^2\dive u|^2+|\nabla^4n|^2)dx+\frac{P'(1)\delta}{2}\|\nabla^2\r\|_{L^2}^2\\&
\le \delta C_2(\|\nabla^2u\|_{L^2}^2+\|\nabla^3n\|_{L^2}^2)
\end{aligned}
\end{equation}
for all $t\ge T_1$, where $C_2$ is a constant independent of time
and $\delta$ is a small constant.
\end{lemm}
\begin{proof}
Multiplying $\delta$ to \eqref{p2} and summing with \eqref{p1}, choosing $\delta$ small enough and using the uniform estimate \eqref{uni}, we get
\begin{equation*}
\begin{aligned}
&\frac{d}{dt}X_h(t)+\int(\frac{3\mu}{4}|\nabla^3u|^2+(\mu+\lambda)|\nabla^2\dive u|^2+|\nabla^4n|^2)dx+\frac{3P'(1)\delta}{4}\|\nabla^2\r\|_{L^2}^2\\&\le
C(\|(\r,u,\nabla u,\nabla n)\|_{L^2}^\frac14+\|(\r,u,\nabla n)\|_{H^1}+\|\nabla^2n\|_{L^2}^\frac12)(\|\nabla^2\r\|_{L^2}^2+\|\nabla^3u\|_{L^2}^2+\|\nabla^4n\|_{L^2}^2)\\&
\quad +C(\|(\r,u,\nabla u,\nabla n)\|_{L^2}^\frac14+\|(u,\nabla n)\|_{H^1}+\|\nabla^2n\|_{L^2}^\frac12)\|\nabla^3n\|_{L^2}^2\\&
\quad +C\delta\|\nabla^2u\|_{L^2}^2+C_1\delta\|\nabla^3n\|_{L^2}^2.
\end{aligned}
\end{equation*}
According to the decay result \eqref{decay}, one may conclude that
\begin{equation*}
\|(\r,u,\nabla u,\nabla n)\|_{L^2}^\frac14+\|(\r,u,\nabla n)\|_{H^1}+\|\nabla^2n\|_{L^2}^\frac12\le C(1+t)^{-\frac{3}{16}}.
\end{equation*}
Thus, there exists a large time $T_1>0$ such that
\begin{equation*}
C\|(\r,u,\nabla u,\nabla n)\|_{L^2}^\frac14+\|(\r,u,\nabla n)\|_{H^1}+\|\nabla^2n\|_{L^2}^\frac12\le \frac14\min\{\mu, 3, P'(1)
\delta,  4C_1\delta\}
\end{equation*}
holds on for all $t \ge T_1$. Therefore, we obtain that
\begin{equation*}
\begin{aligned}
\frac{d}{dt}X_h(t)+\frac12\int(\mu|\nabla^3u|^2
+(\mu+\lambda)|\nabla^2\dive u|^2+|\nabla^4n|^2)dx+\frac{P'(1)\delta}{2}\|\nabla^2\r\|_{L^2}^2
\le \delta C_2(\|\nabla^2u\|_{L^2}^2+\|\nabla^3n\|_{L^2}^2),
\end{aligned}
\end{equation*}
which completes the proof of this lemma.
\end{proof}

\subsection{Cancellation of a low-medium-frequency part}

\quad
In this subsection, based on the second order energy estimate \eqref{e1}, we get $L^\infty_tL^2_x-$ norm estimate on $\nabla^2(\r, u, \nabla n)$ by removing the low-medium-frequency part of the term $\int\nabla u\cdot\nabla^2\r dx$.

\begin{lemm}\label{lemma1}
It holds that
\begin{equation*}
\begin{aligned}
\|\nabla^2(\r,u,\nabla n)(t)\|_{L^2}^2
\le Ce^{-C_4t}\|\nabla^2(\r,u,\nabla n)(T_1)\|_{L^2}^2
 +C\int_{T_1}^te^{-C_4(t-\tau)}\|\nabla^2(\r^L,u^L,\nabla n^L)(\tau)\|_{L^2}^2d\tau,
\end{aligned}
\end{equation*}
where the positive constant $C_4$ is independent of time, $T_1$ is the large time given in Lemma \ref{t1}.
\end{lemm}
\begin{proof}
Applying $\nabla$ on the second equation of \eqref{liqu}, then taking $L^2$ inner product with $\nabla^2\r^L$ (see \eqref{decom}), we obtain that
\begin{equation}\label{y1}
\begin{split}
\int\nabla u_t\cdot\nabla^2\r^Ldx
=\int(\mu\Delta\nabla u+(\mu+\lambda)\nabla^2\dive u)\cdot\nabla^2\r^Ldx-P'(1)\int\nabla^2\r\cdot\nabla^2\r^Ldx+\int\nabla S_2\cdot\nabla^2\r^L.
\end{split}
\end{equation}
Integrating by part and using $\eqref{liqu1}_1$, we get
\begin{equation}\label{y2}
\begin{aligned}
\int\nabla u_t\cdot\nabla^2\r^Ldx&=\frac{d}{dt}\int\nabla u\cdot\nabla^2\r^Ldx-\int\nabla u\cdot\nabla^2\r^L_tdx\\&
=\frac{d}{dt}\int\nabla u\cdot\nabla^2\r^Ldx+\int\nabla \dive u\cdot\nabla\r^L_tdx\\&
=\frac{d}{dt}\int\nabla u\cdot\nabla^2\r^Ldx+\int\nabla\dive u\cdot(\nabla S_1^L-\nabla\dive u^L)dx.
\end{aligned}
\end{equation}
Therefore, substituting \eqref{y2} into \eqref{y1}, and using the H\"{o}lder and Cauchy inequalities, we get
\begin{equation*}
\begin{aligned}
-\frac{d}{dt}\int\nabla u\cdot\nabla^2\r^Ldx&\le \frac{\mu}{2}\|\nabla^3u\|_{L^2}^2+\frac{\mu+\lambda}{2}\|\nabla^2\dive u\|_{L^2}^2+\frac{P'(1)}{4}\|\nabla^2\r\|_{L^2}^2
+\frac12\|\nabla S_1^L\|_{L^2}^2+\frac12\|\nabla S_2\|_{L^2}^2\\&
\quad +\|\nabla\dive u\|_{L^2}^2+\frac12\|\nabla\dive u^L\|_{L^2}^2+C\|\nabla^2\r^L\|_{L^2}^2.
\end{aligned}
\end{equation*}
By routine calculation, it is easy to see that
\begin{equation}\label{y3}
\begin{aligned}
\|\nabla S_1\|_{L^2}&\le \|\nabla\r\|_{L^3}\|\dive u\|_{L^6}+\|\r\|_{L^\infty}\|\nabla\dive u\|_{L^2}+\|u\|_{L^\infty}\|\nabla^2\r\|_{L^2}+\|\nabla\r\|_{L^3}\|\nabla u\|_{L^6}\\&
\le C(\|\nabla\r\|_{L^3}+\|\r\|_{L^\infty}+\|u\|_{L^\infty})(\|\nabla^2u\|_{L^2}+\|\nabla^2\r\|_{L^2})\\&
\le C(\|\r\|_{L^2}^{\frac14}+\|u\|_{L^2}^{\frac14})(\|\nabla^2u\|_{L^2}+\|\nabla^2\r\|_{L^2}),
\end{aligned}
\end{equation}
where we used \eqref{b4} in the last inequality. According to the frequency decomposition \eqref{decom}, and using Lemma \ref{A1}, we have
\begin{equation*}
\|\nabla S_1^L\|_{L^2}\le \|\nabla S_1\|_{L^2}+\|\nabla S_1^h\|_{L^2}\le C\|\nabla S_1\|_{L^2},
\end{equation*}
this together with \eqref{d2} and \eqref{y3} gives rise to
\begin{equation*}
\|\nabla S_1^L\|_{L^2}^2+\|\nabla S_2\|_{L^2}^2\le C(\|(\r, u, \nabla n)\|_{L^2}^{\frac12}+\|(u,\nabla n)\|_{H^1}^2)(\|\nabla^2u\|_{H^1}^2+\|\nabla^2\r\|_{L^2}^2+\|\nabla^3n\|_{H^1}^2).
\end{equation*}
Then we have
\begin{equation}\label{e2}
\begin{aligned}
&-\frac{d}{dt}\int\nabla u\cdot\nabla^2\r^Ldx\\&
\le \frac{\mu}{2}\|\nabla^3u\|_{L^2}^2+\frac{\mu+\lambda}{2}\|\nabla^2\dive u\|_{L^2}^2+\frac{P'(1)}{4}\|\nabla^2\r\|_{L^2}^2+\|\nabla\dive u\|_{L^2}^2\\&
\quad +\frac12\|\nabla\dive u^L\|_{L^2}^2+C\|\nabla^2\r^L\|_{L^2}^2\\&
\quad +C(\|(\r, u, \nabla n)\|_{L^2}^{\frac12}+\|(u,\nabla n)\|_{H^1}^2)(\|\nabla^2u\|_{H^1}^2+\|\nabla^2\r\|_{L^2}^2+\|\nabla^3n\|_{H^1}^2).
\end{aligned}
\end{equation}
Adding \eqref{e1} and $\delta\times\eqref{e2}$, and using the uniform estimate \eqref{uni}, we get
\begin{equation}\label{e3}
\begin{aligned}
&\frac{d}{dt}\big(X_h(t)-\delta\int\nabla u\cdot\nabla^2\r^Ldx\big)+\frac12\int(\mu|\nabla^3u|^2+(\mu+\lambda)|\nabla^2\dive u|^2+|\nabla^4n|^2)dx\\&
\quad +\frac{P'(1)\delta}{2}\|\nabla^2\r\|_{L^2}^2\\&
\le \delta C_2(\|\nabla^2u\|_{L^2}^2+\|\nabla^3n\|_{L^2}^2)+
\frac{\mu}{2}\delta\|\nabla^3u\|_{L^2}^2+\frac{\mu+\lambda}{2}\delta\|\nabla^2\dive u\|_{L^2}^2\\&
\quad+\frac{P'(1)}{4}\delta\|\nabla^2\r\|_{L^2}^2+\delta\|\nabla\dive u\|_{L^2}^2 +\frac12\delta\|\nabla\dive u^L\|_{L^2}^2+C\delta\|\nabla^2\r^L\|_{L^2}^2\\&
\quad+\delta C_3(\|\nabla^2u\|_{H^1}^2+\|\nabla^2\r\|_{L^2}^2+\|\nabla^3n\|_{H^1}^2).
\end{aligned}
\end{equation}
Using Lemma \ref{A1}, we have
\begin{equation}\label{e4}
\begin{aligned}
\frac{\mu}{2}\|\nabla^3u\|_{L^2}^2+\frac12\|\nabla^4n\|_{L^2}^2
\ge \frac{\mu}{4}R_0^2\|\nabla^2u^h\|_{L^2}^2
+\frac{\mu}{4}\|\nabla^3u\|_{L^2}^2+\frac14R^2_0\|\nabla^3n^h\|_{L^2}^2
+\frac14\|\nabla^4n\|_{L^2}^2.
\end{aligned}
\end{equation}
Substituting \eqref{e4} into \eqref{e3}, and adding $ \frac{\mu}{4}R_0^2\|\nabla^2u^L\|_{L^2}^2+\frac14R^2_0\|\nabla^3n^L\|_{L^2}^2$ on both side of the resulting inequality, we get
\begin{equation*}
\begin{aligned}
&\frac{d}{dt}\big(X_h(t)-\delta\int\nabla u\cdot\nabla^2\r^Ldx\big)+\frac{\mu}{8}R^2_0\|\nabla^2u\|_{L^2}^2+\frac{\mu}{4}\|\nabla^3u\|_{L^2}^2+\frac{\mu+\lambda}{2}\|\nabla^2\dive u\|_{L^2}^2\\&
\quad +\frac{R^2_0}{8}\|\nabla^3m\|_{L^2}^2+\frac14\|\nabla^4m\|_{L^2}^2+\frac{P'(1)\delta}{2}\|\nabla^2\r\|_{L^2}^2\\&
\le \delta C_2(\|\nabla^2u\|_{L^2}^2+\|\nabla^3n\|_{L^2}^2)+
\frac{\mu}{2}\delta\|\nabla^3u\|_{L^2}^2+\frac{\mu+\lambda}{2}\delta\|\nabla^2\dive u\|_{L^2}^2\\&
\quad+\frac{P'(1)}{4}\delta\|\nabla^2\r\|_{L^2}^2+\delta\|\nabla\dive u\|_{L^2}^2
+\frac12\delta\|\nabla\dive u^L\|_{L^2}^2+C\delta\|\nabla^2\r^L\|_{L^2}^2\\&
\quad+\delta C_3(\|\nabla^2u\|_{H^1}^2+\|\nabla^2\r\|_{L^2}^2+\|\nabla^3n\|_{H^1}^2).
\end{aligned}
\end{equation*}
Choosing $\delta\le\min\{\frac18,\frac{\mu}{16C_3},\frac{1}{8C_3}\}$, $R_0^2\ge \max\{\frac{6C_2}{\mu},\frac{3}{\mu},\frac{6C_3}{\mu},4C_2,4C_3\}$, then we have
\begin{equation}\label{e5}
\begin{aligned}
&\frac{d}{dt}\big(X_h(t)-\delta\int\nabla u\cdot\nabla^2\r^Ldx\big)+\frac{\mu}{16}R^2_0\|\nabla^2u\|_{L^2}^2+\frac{\mu}{8}\|\nabla^3u\|_{L^2}^2+\frac{\mu+\lambda}{4}\|\nabla^2\dive u\|_{L^2}^2\\&
\quad +\frac{R^2_0}{16}\|\nabla^3n\|_{L^2}^2+\frac18\|\nabla^4n\|_{L^2}^2+\frac{P'(1)\delta}{4}\|\nabla^2\r\|_{L^2}^2\\&
\le C\|\nabla^2(\r^L,u^L,\nabla n^L)\|_{L^2}^2.
\end{aligned}
\end{equation}
According to the decomposition \eqref{decom}, we have
\begin{equation*}
\begin{aligned}
X_h(t)-\delta\int\nabla u\cdot\nabla^2\r^Ldx
=\frac12(\|\nabla^2\r\|_{L^2}^2+\|\nabla^2u\|_{L^2}^2+\|\nabla^3n\|_{L^2}^2)+\delta\int\nabla u\cdot\nabla^2\r^hdx.
\end{aligned}
\end{equation*}
Integrating by part, and using Lemma \ref{A1}, we have
\begin{equation*}
\begin{aligned}
\delta\int\nabla u\cdot\nabla^2\r^hdx
=-\delta\int\nabla\dive u\cdot\nabla\r^hdx
\le \frac{\delta}{2}\|\nabla\r^h\|_{L^2}^2+\frac{\delta}{2}\|\nabla\dive u\|_{L^2}^2
\le  \frac{\delta}{2}\|\nabla^2\r\|_{L^2}^2+\frac{\delta}{2}\|\nabla\dive u\|_{L^2}^2,
\end{aligned}
\end{equation*}
which implies that
\begin{equation}\label{e6}
X_h(t)-\delta\int\nabla u\cdot\nabla^2\r^Ldx\sim\|\nabla^2(\r,u,\nabla n)\|_{L^2}^2,
\end{equation}
where we have used the fact that $0<\delta\le \frac18$. Thanks to \eqref{e5} and \eqref{e6}, there exists a constant $C_4$ such that
\begin{equation}\label{e7}
\frac{d}{dt}\big(X_h(t)-\delta\int\nabla u\cdot\nabla^2\r^Ldx\big)+C_4\big(X_h(t)-\delta\int\nabla u\cdot\nabla^2\r^Ldx\big)\le C\|\nabla^2(\r^L,u^L,\nabla n^L)\|_{L^2}^2.
\end{equation}
Multiplying \eqref{e7} by $e^{C_4t}$ and integrating
with respect to time over $[T_1,t]$, we get
\begin{equation*}
\begin{aligned}
&\quad X_h(t)-\delta\int\nabla u\cdot\nabla^2\r^Ldx\\
&\le e^{-C_4t}\big(X_h(T_1)-\delta\int\nabla u(T_1)\cdot\nabla^2\r^L(T_1)dx\big)
     +C\int_{T_1}^te^{-C_4(t-\tau)}\|\nabla^2(\r^L,u^L,\nabla n^L)(\tau)\|_{L^2}^2d\tau.
\end{aligned}
\end{equation*}
Using the equivalent equation \eqref{e6} again, we complete the proof of this lemma.
\end{proof}
\subsection{Decay estimates of the low-medium-frequency part}

\quad
In this subsection, based on the classical semigroup method and the $L^2-$ norm decay estimate for spectral analysis on the linearized system, we obtain the estimate of the low-medium-frequency part of the solution to the Cauchy problem \eqref{liqu1}.
In order to get the decay estimate of $\|\nabla n\|_{H^2}$, we applying $\nabla$ operator to the third equation of system \eqref{liqu1}, then \eqref{liqu1} becomes
\begin{equation}\label{liqu0}
  	\left\{\begin{aligned}
	&\p_t\r+\dive u=S_1,\\
	&\p_t u-\mu\Delta u-(\mu+\lambda)\nabla\dive u+P'(1)\nabla\r=S_2,\\
	&\p_t \nabla n-\Delta \nabla n=\nabla S_3,
    	\end{aligned}\right.
  \end{equation}
Denote
\begin{equation*}
\mathbb U(t):=(\r(t),u(t), \nabla n(t))^T,
\end{equation*}
and the differential operator $\mathbb G$:
\begin{gather*}
\mathbb G=
\begin{pmatrix}
0            & \dive  & 0 \\
-P'(1)\nabla & -\mu\Delta-(\mu+\lambda)\nabla\dive & 0\\
0 &  0 & -\Delta
\end{pmatrix}
\end{gather*}
Then we can rewrite the system \eqref{liqu0} as
\begin{equation}\label{eq2}
  	\left\{\begin{aligned}
	&\p_t\mathbb U+\mathbb G\mathbb U=S(\mathbb U),\\
	&\mathbb U|_{t=0}=\mathbb U(0),
    	\end{aligned}\right.
  \end{equation}
  where
  \begin{equation}\label{eq3}
  S(\mathbb U):=(S_1,S_2,\nabla S_3)^T, \quad \mathbb U(0):=(\r_0, u_0, \nabla n_0).
  \end{equation}
Moreover, we define
\begin{equation*}
\overline{\mathbb U}(t):=(\bar\r(t),\bar u(t),\nabla \bar n(t))^T,
\end{equation*}
then we have the following corresponding linearized problem
\begin{equation}\label{eq1}
  	\left\{\begin{aligned}
	&\p_t\overline{\mathbb U}+\mathbb G\overline{\mathbb U}=0,\\
	&\overline{\mathbb U}|_{t=0}=\mathbb U(0).
    	\end{aligned}\right.
  \end{equation}
  Taking the Fourier transform on \eqref{eq1} with respect to space variable and solving the ODE, we get
  \begin{equation*}
  \overline{\mathbb U}(t)=\mathcal G(t)\mathbb U (0),
  \end{equation*}
  where $\mathcal G(t)=e^{-t\mathbb G}(t\ge0)$ is the semigroup generated by the operator $\mathbb G$ and $\mathcal G(t)f:=\mathcal F^{-1}(e^{-t\mathbb G_\xi}\hat f(\xi))$ with
  \begin{gather*}
\mathbb G_\xi=
\begin{pmatrix}
0            & i\xi^T  & 0 \\
i\xi & \mu|\xi|^2\delta_{ij}+(\mu+\lambda)\xi_i\xi_j & 0\\
0 &  0 & |\xi|^2
\end{pmatrix}
\end{gather*}

Next, according to the decay estimate of solution to the linearized system \eqref{eq1} in frequency regimes (see Appendix \ref{appendixA}), we give the following estimate of the low-medium-frequency part of the solution.
\begin{lemm}
Assume $1\le p\le 2$, for any integer $k\ge0$, there holds
\begin{equation*}
\|\nabla^k(\mathcal G(t)\mathbb U^L(0))\|_{L^2}\le C(1+t)^{-\frac32(\frac1p-\frac12)-\frac k2}\|\mathbb U(0)\|_{L^p}.
\end{equation*}
\end{lemm}
\begin{proof}
Set $b:=\Lambda^{-1}\dive u$ be the ``compressible part", $\mathcal Pu:=\Lambda^{-1}\text{curl} u$ be the ``incompressible part", where $\Lambda^2=\Delta$, then $u=-\Lambda^{-1}\nabla b-\Lambda^{-1}\dive\mathcal Pu$.
Using the Plancherel theorem,  \eqref{g1} and \eqref{g2}, we obtain that
\begin{equation*}
\begin{aligned}
\|\p_x^\alpha(\bar\r^L,\bar b^L, \nabla\bar n^L)(t)\|_{L^2}^2&=\|(i\xi)^\alpha(\bar\r^L,\bar b^L,  \nabla \bar n^L)\|_{L^2_\xi}^2\\&
=\int\Big|(i\xi)^\alpha(\widehat{\bar\r^L},\widehat{\bar b^L},\widehat{\nabla \bar n^L})(t,\xi)\Big|^2d\xi\\&
\le C\int_{|\xi|\le R_0}|\xi|^{2\alpha}|(\widehat{\bar\r},\widehat{\bar b},\widehat{\nabla\bar n})(t,\xi)|^2d\xi\\&
\le C\int_{|\xi|\le r_0}|\xi|^{2\alpha}e^{-C_5|\xi|^2t}|(\widehat{\bar\r},\widehat{\bar b},\widehat{\nabla\bar n})(0,\xi)|^2d\xi\\&
\quad +\int_{r_0< |\xi|\le R_0}|\xi|^{2\alpha}e^{-\kappa t}|(\widehat{\bar\r},\widehat{\bar b},\widehat{\nabla\bar n})(0,\xi)|^2d\xi.
\end{aligned}
\end{equation*}
Using the H\"{o}lder, Hausdorff-Young inequalities, we have
\begin{equation}\label{g6}
\begin{aligned}
\|\p_x^\alpha(\bar\r^L,\bar b^L,\nabla\bar n^L)(t)\|_{L^2}
\le C\|(\hat\r,\hat b,\widehat{\nabla n})(0)\|_{L^q_\xi}
    (1+t)^{-\frac32(\frac12-\frac1q)-\frac{|\alpha|}{2}}
\le C\|(\r, b,\nabla n)(0)\|_{L^p}(1+t)^{-\frac32(\frac1p-\frac12)-\frac{|\alpha|}{2}},
\end{aligned}
\end{equation}
where $1\le p\le 2\le q\le +\infty$, $\frac1p+\frac1q=1$.
Similarly, according to \eqref{g3}, we have
\begin{equation}\label{g5}
\begin{aligned}
\|\p_x^\alpha(\overline{\mathcal P u})^L(t)\|_{L^2}
&\le C\Big(\int_{|\xi|\le R_0}|\xi|^{2\alpha}|\widehat{\mathcal Pu}(t,\xi)|^2d\xi\Big)^\frac12
\le C\Big(\int_{|\xi|\le R_0}e^{-2\mu|\xi|^2t}|\widehat{\mathcal Pu}(0,\xi)|^2d\xi\Big)\\
&\le C\|u_0\|_{L^p}(1+t)^{-\frac32(\frac1p-\frac12)-\frac{|\alpha|}{2}}.
\end{aligned}
\end{equation}
The combination of \eqref{g6} and \eqref{g5} completes the proof of this lemma.
\end{proof}

Next, we establish the decay estimates of the solution to the nonlinear problem \eqref{eq2}-\eqref{eq3}. According to Duhamel principle, we rewrite the solution of system \eqref{eq2} as follows
\begin{equation*}
\mathbb U(t)=\mathcal G(t)\mathbb U(0)+\int_0^t\mathcal G(t-\tau)S(\mathbb U)(\tau)d\tau.
\end{equation*}
Then we get the following estimates on the low-medium-frequency part of the solution to the nonlinear problem \eqref{eq2}-\eqref{eq3}.
\begin{lemm}\label{low-me}
For any integer $k\ge 0$, it holds true
\begin{equation*}
\begin{aligned}
\|\nabla^k\mathbb U^L(t)\|_{L^2}
&\le C_6(1+t)^{-\frac34-\frac k2}\|\mathbb U(0)\|_{L^1}
+C_6\int_0^{\frac t2}(1+t-\tau)^{-\frac34-\frac k2}\|S(\mathbb U)(\tau)\|_{L^1}d\tau\\
&+C_6\int_{\frac t2}^t(1+t-\tau)^{-\frac k2}\|S(\mathbb U)(\tau)\|_{L^2}d\tau,
\end{aligned}
\end{equation*}
where the positive constant $C_6$ independent of time.
\end{lemm}

\subsection{Decay rate for the nonlinear system}

\quad
In this subsection, we will establish the time decay rate of
the solution to the original nonlinear problem \eqref{liqu0}.
\begin{lemm}\label{lemm28}
Under the assumptions of Theorem \ref{th1}, there exists a positive constant $T_2$, such that the global solution $(\r,u,n)$ of Cauchy problem \eqref{liqu1} has the estimate
\begin{equation}\label{decay1}
\|\nabla^k(\r,u,\nabla n)(t)\|_{L^2}
\le C(1+t)^{-\frac34-\frac k2}, \quad k=0,1,2,
\end{equation}
for all $t\ge T_2$. Here $C$ is a positive constant independent of time.
\end{lemm}
\begin{proof}
Let us denote
\begin{equation}\label{i0}
M(t):=\sup_{0\le\tau\le t}\sum_{l=0}^2(1+\tau)^{\frac34+\frac l2}\|\nabla^l(\r,u,\nabla n)(\tau)\|_{L^2},
\end{equation}
and hence, we have for $0\le l\le 2$
\begin{equation}\label{h5}
\|\nabla^l(\r,u,\nabla n)(\tau)\|_{L^2}\le C(1+\tau)^{-\frac34-\frac l2}M(t), \quad 0\le\tau\le t.
\end{equation}
According to the definition of $S(\mathbb U)$, we get
\begin{equation}\label{h1}
\|S(\mathbb U)(\tau)\|_{L^1}\le \|S_1(\tau)\|_{L^1}+\|S_2(\tau)\|_{L^1}+\|\nabla S_3(\tau)\|_{L^1}.
\end{equation}
By routine checking, it is easy to see
\begin{equation}\label{h2}
\|S_1\|_{L^1}\le \|(\r, u)\|_{L^2}\|\nabla(\r,u)\|_{L^2}.
\end{equation}
By the H\"{o}lder and Sobolev inequalities, we have
\begin{equation}\label{h3}
\begin{aligned}
\|S_2\|_{L^1}&\le \|u\|_{L^2}\|\nabla u\|_{L^2}+\|\r\|_{L^2}\|\nabla^2u\|_{L^2}+\|\r\|_{L^2}\|\nabla\r\|_{L^2}+\|\nabla n\|_{L^2}\|\nabla^2n\|_{L^2}\\&
\le C\|(\r,u,\nabla n)\|_{L^2}\|\nabla(\r,u,\nabla n)\|_{L^2}+\|\r\|_{L^2}\|\nabla u\|_{L^2}^{\frac12}\|\nabla^3u\|_{L^2}^{\frac12},
\end{aligned}
\end{equation}
and
\begin{equation}\label{h4}
\begin{aligned}
\|\nabla S_3\|_{L^1}&\le \|\nabla u\|_{L^2}\|\nabla n\|_{L^2}+\|u\|_{L^2}\|\nabla^2n\|_{L^2}+\|\nabla n\|_{L^2}\|\nabla^2n\|_{L^2}(\|n\|_{L^\infty}+\underline d)\\&
\quad +\|\nabla n\|_{L^2}\|\nabla n\|_{L^3}\|\nabla n\|_{L^6}\\&
\le \|\nabla u\|_{L^2}\|\nabla n\|_{L^2}+\|u\|_{L^2}\|\nabla^2n\|_{L^2}+\|\nabla n\|_{L^2}^{\frac32}\|\nabla^2n\|_{L^2}^{\frac32}+\|\nabla n\|_{L^2}\|\nabla^2n\|_{L^2},
\end{aligned}
\end{equation}
where we used interpolation inequalities in the last inequality as follows
\begin{equation*}
\|n\|_{L^\infty}+\|\nabla n\|_{L^3}\le \|\nabla n\|_{L^2}^{\frac12}\|\nabla^2n\|_{L^2}^{\frac12}.
\end{equation*}
Adding \eqref{h2}, \eqref{h3} and \eqref{h4} into \eqref{h1}, and using the decay estimate \eqref{decay}, we get
\begin{equation}\label{h6}
\|S(\mathbb U)(\tau)\|_{L^1}\le C(1+\tau)^{-\frac32}+\|\r\|_{L^2}\|\nabla u\|_{L^2}^{\frac12}\|\nabla^3u\|_{L^2}^{\frac12}.
\end{equation}
Next, we need estimate $\|S(\mathbb U)(\tau)\|_{L^2}$. According to the definition of $S(\mathbb U)$, we get
\begin{equation*}
\|S(\mathbb U)(\tau)\|_{L^2}\le \|S_1(\tau)\|_{L^2}+\|S_2(\tau)\|_{L^2}+\|\nabla S_3(\tau)\|_{L^2}.
\end{equation*}
Using the decay estimate \eqref{decay} and \eqref{h5}, we get
\begin{equation*}
\|S_1(\tau)\|_{L^2}\le C\|(\r,u)\|_{H^1}\|\nabla^2(\r,u)\|_{L^2}\le (1+\tau)^{-\frac{10}{4}}M(t).
\end{equation*}
Using the H\"{o}lder, Sobolev inequalities, the uniform estimate \eqref{uni}
and decay estimate \eqref{decay}, we have
\begin{equation*}
\begin{aligned}
\|S_2(\tau)\|_{L^2}
&\le \|u\|_{L^3}\|\nabla u\|_{L^6}+\|\r\|_{L^\infty}\|\nabla^2u\|_{L^2}
  +\|\r\|_{L^3}\|\nabla\r\|_{L^6}+\|\nabla n\|_{L^3}\|\nabla^2n\|_{L^6}
\\&\le \big(\|(\r,u,\nabla n)\|_{H^1}
  +\|\nabla\r\|_{L^2}^\frac12\|\nabla^2\r\|_{L^2}^\frac12\big)
   \|\nabla^2(\r,u,\nabla n)\|
\\&
\le C\big((1+\tau)^{-\frac34}+(1+\tau)^{-\frac38}\big)(1+\tau)^{-\frac74}M(t)\\&
\le C(1+\tau)^{-\frac{17}{8}}M(t).
\end{aligned}
\end{equation*}
Similarly, we have
\begin{equation*}
\begin{aligned}
\|\nabla S_3(\tau)\|_{L^2}&\le \|\nabla n\|_{L^3}\|\nabla u\|_{L^6}+\|u\|_{L^3}\|\nabla^2n\|_{L^6}+\|\nabla n\|_{L^3}\|\nabla^2n\|_{L^6}(\|n\|_{L^\infty}+\underline d)\\&
\quad +\|\nabla n\|_{L^\infty}\|\nabla n\|_{L^3}\|\nabla n\|_{L^6}\\&\le
\|\nabla n\|_{H^1}\|\nabla^2u\|_{L^2}+\|u\|_{H^1}\|\nabla^3n\|_{L^2}+\|\nabla n\|_{H^1}\|\nabla^3n\|_{L^2}(\|n\|_{L^\infty}+\underline d)\\&
\quad +\|\nabla n\|_{L^\infty}\|\nabla n\|_{H^1}\|\nabla^2 n\|_{L^2}.
\end{aligned}
\end{equation*}
By the Sobolev inequality, the uniform estimate \eqref{uni}, and \eqref{h5},  we get
\begin{equation*}
\|n\|_{L^\infty} \le C\|\nabla n\|_{L^2}^\frac12\|\nabla^2n\|_{L^2}^\frac12\le C,
\end{equation*}
and
\begin{equation*}
\|\nabla n\|_{L^\infty}
\le C\|\nabla^2n\|_{L^2}^\frac12\|\nabla^3n\|_{L^2}^\frac12
\le (1+\tau)^{-\frac58}M(t)^\frac12(1+\tau)^{-\frac78}M(t)^\frac12
\le C(1+\tau)^{-\frac{12}{8}}M(t).
\end{equation*}
This together with the decay estimate \eqref{decay} yields
\begin{equation*}
\begin{aligned}
\|\nabla S_3(\tau)\|_{L^2}&\le C \|(u,\nabla n)\|_{H^1}\|\nabla^2(u,\nabla n)\|_{L^2}+C(1+\tau)^{-\frac{12}{8}}M(t)\|\nabla n\|_{H^1}\|\nabla^2 n\|_{L^2}
\\&
\le C(1+\tau)^{-\frac{10}{4}}M(t)+C(1+\tau)^{-3}M(t)\\&
\le C(1+\tau)^{-\frac{10}{4}}M(t).
\end{aligned}
\end{equation*}
Collecting the above estimates, we obtain that
\begin{equation}\label{h7}
\|S(\mathbb U)(\tau)\|_{L^2}\le C(1+\tau)^{-\frac{17}{8}}M(t).
\end{equation}
By Lemma \ref{low-me}, \eqref{h6} and \eqref{h7}, we have for $0\le k\le 2$
\begin{equation}\label{i1}
\begin{aligned}
\|\nabla^k\mathbb U^L(t)\|_{L^2}&\le C_5(1+t)^{-\frac34-\frac k2}\|\mathbb U(0)\|_{L^1}+C_5\int_{\frac t2}^t(1+t-\tau)^{-\frac k2}(1+\tau)^{-\frac{17}{8}}M(t)d\tau\\&
\quad +C_5\int_0^{\frac t2}\big((1+\tau)^{-\frac32}+\|\r\|_{L^2}\|\nabla u\|_{L^2}^\frac12\|\nabla^3u\|_{L^2}^\frac12\big)(1+t-\tau)^{-\frac34-\frac k2}d\tau.
\end{aligned}
\end{equation}
Direct calculation gives rise to
\begin{equation}\label{i2}
\int_{\frac t2}^t(1+t-\tau)^{-\frac k2}(1+\tau)^{-\frac{17}{8}}M(t)d\tau
\le C(1+t)^{-\frac98-\frac k2}M(t).
\end{equation}
and
\begin{equation}\label{i22}
\int_0^{\frac t2}(1+\tau)^{-\frac32}(1+t-\tau)^{-\frac34-\frac k2}d\tau
\le C(1+t)^{-\frac34-\frac k2}.
\end{equation}
Using the Young inequality, the decay estimate \eqref{decay} and the uniform estimate \eqref{uni}, we get
\begin{equation}\label{i3}
\begin{aligned}
&\int_0^{\frac t2}(1+t-\tau)^{-\frac34-\frac k2}\|\r\|_{L^2}\|\nabla u\|_{L^2}^\frac12\|\nabla^3u\|_{L^2}^\frac12d\tau\\&
\le \int_0^{\frac t2}(1+t-\tau)^{-\frac34-\frac k2}\big(\|\r\|_{L^2}^2+\|\nabla u\|_{L^2}^2+\|\nabla^3u\|_{L^2}^2\big)d\tau\\&
\le C(1+t)^{-\frac34-\frac k2}\int_0^{\frac t2}\big((1+\tau)^{-\frac32}+\|\nabla^3u(\tau)\|_{L^2}^2\big)d\tau\\&
\le C(1+t)^{-\frac34-\frac k2}.
\end{aligned}
\end{equation}
Substituting  the estimates \eqref{i2}, \eqref{i22} and \eqref{i3} into \eqref{i1}, we get
\begin{equation}\label{i4}
\begin{aligned}
\|\nabla^k\mathbb U^L(t)\|_{L^2}\le C\big(\|\mathbb U(0)\|_{L^1}+C+(1+t)^{-\frac38}M(t)\big)(1+t)^{-\frac34-\frac k2},
\end{aligned}
\end{equation}
which together with Lemma \ref{lemma1} yields
\begin{equation}\label{i4}
\begin{aligned}
\|\nabla^2\mathbb U(t)\|_{L^2}^2
&\le Ce^{-C_4t}\|\nabla^2\mathbb U(T_1)\|_{L^2}^2
   +C\big(\|\mathbb U(0)\|^2_{L^1}+1\big)\int_{T_1}^te^{-C_4(t-\tau)}(1+\tau)^{-\frac72}d\tau\\&
\quad +C\int_{T_1}^te^{-C_4(t-\tau)}(1+\tau)^{-\frac34}M^2(\tau)(1+\tau)^{-\frac72}d\tau\\&
\le Ce^{-C_4t}\|\nabla^2\mathbb U(T_1)\|_{L^2}^2+C\big(\|\mathbb U(0)\|^2_{L^1}+1\big)(1+t)^{-\frac72}
  +CM^2(t)(1+t)^{-\frac{17}{4}}\\&
\le Ce^{-C_4t}\|\nabla^2\mathbb U(T_1)\|_{L^2}^2
+C\big(\|\mathbb U(0)\|^2_{L^1}+1+(1+t)^{-\frac34}M^2(t)\big)(1+t)^{-\frac72}
\end{aligned}
\end{equation}
for $t\ge T_1$. Therefore, using the decomposition \eqref{decom} and Lemma \ref{A1}, for $0\le k\le 2$, we obtain that
\begin{equation}\label{i5}
\|\nabla^k\mathbb U(t)\|_{L^2}^2
\le C\|\nabla^k\mathbb U^L(t)\|_{L^2}^2+ C\|\nabla^k\mathbb U^h(t)\|_{L^2}^2
\le C\|\nabla^k\mathbb U^L(t)\|_{L^2}^2+ C\|\nabla^2\mathbb U(t)\|_{L^2}^2.
\end{equation}
Putting \eqref{i3} and \eqref{i4} into \eqref{i5}, for all $t\ge T_1$, we obtain that
\begin{equation}\label{i6}
\begin{aligned}
\|\nabla^k\mathbb U(t)\|_{L^2}&\le C\big(\|\mathbb U(0)\|^2_{L^1}+1+(1+t)^{-\frac34}M^2(t)\big)(1+t)^{-\frac32-k}
      +Ce^{-C_4t}\|\nabla^2\mathbb U(T_1)\|_{L^2}^2\\&
\quad +C\big(\|\mathbb U(0)\|^2_{L^1}+1+(1+t)^{-\frac34}M^2(t)\big)(1+t)^{-\frac72}.
\end{aligned}
\end{equation}
Recalling the definition of $M(t)$, from \eqref{i6}, we know there exists a positive constant $C_6$ such that
\begin{equation*}
M^2(t)\le C_6\big(\|\mathbb U(0)\|_{L^1}^2+1+(1+t)^{-\frac34}M^2(t)+\|\nabla^2\mathbb U(T_1)\|_{L^2}^2\big).
\end{equation*}
Choosing $T_2$, such that for all $t\ge T_2$, there holds
\begin{equation*}
C_6(1+t)^{-\frac34}\le \frac12.
\end{equation*}
Then we have
\begin{equation*}
M^2(t)\le 2C_6\big(\|\mathbb U(0)\|_{L^1}^2+1+\|\nabla^2\mathbb U(T_1)\|_{L^2}^2\big), \quad t\ge T_2,
\end{equation*}
which together with the uniform estimate \eqref{uni} implies
\begin{equation*}
M(t)\le C, \quad \text{for \ all}\  t\in[T_2,+\infty).
\end{equation*}
By the definition of $M(t)$ in \eqref{i0}, we complete the proof of this lemma.
\end{proof}
\textbf{Proof of Theorem \ref{th2}:} With the decay estimates stated in Lemma \ref{lemm28}, we can complete the proof of the Theorem \ref{th2}.

\section{Proof of Theorem \ref{th3}}

\quad
In this section, we will establish the lower bound of decay rate for the global solution of Cauchy problem \eqref{liqu0}.
\begin{lemm}\label{lemm31}
Denote $w_0:=\Lambda n_0$, and
assume that the Fourier transform $\mathcal F(w_0)=\widehat {w_0}$ satisfies $|
\widehat{w_0}|\ge c_0$, $0< |\xi|\ll 1$, with $c_0>0$ a constant. Then the solution $n(t,x)$ obtained in Theorem \ref{th1} has the decay rate for all $t\ge T_3$
\begin{equation*}
\|\nabla ^{k+1} n(t)\|_{L^2}\ge C(1+t)^{-\frac{3+2k}{4}}, \quad k=0,1,2,
\end{equation*}
where $T_3$ is defined below.
\end{lemm}
\begin{proof}
\text{Step 1:} Consider the linearized equation corresponding to the third equation of \eqref{liqu0}:
\begin{equation*}
\p_t\nabla\bar{n}(t,x)-\Delta(\nabla\bar{ n})=0, \quad \nabla\bar{n}(t,x)|_{t=0}=\nabla n_0(x), \quad x\in \mathbb R^3.
\end{equation*}
Using the semigroup method, it is easy to check that
\begin{equation}\label{l1}
\int|\nabla\bar n|^2dx=\int|\widehat{w_0}|^2e^{-2|\xi|^2t}d\xi\ge c_0^2\int_{0<|\xi|\ll1}e^{-2|\xi|^2t}d\xi\ge C(1+t)^{-\frac32}.
\end{equation}
Similarly, we have
\begin{equation}\label{l2}
\int|\nabla^2\bar n|^2dx=\int |\widehat {w_0}|^2|\xi|^2e^{-2|\xi|^2t}d\xi\ge C(1+t)^{-\frac52},
\end{equation}
and
\begin{equation}\label{l3}
\int|\nabla^3\bar n|^2dx=\int |\widehat {w_0}|^2|\xi|^4e^{-2|\xi|^2t}d\xi\ge C(1+t)^{-\frac72}.
\end{equation}

\text{Step 2:} Define $\nabla n_\delta(t,x):=\nabla n(t,x)-\nabla \bar n(t,x)$, then $\nabla n_\delta(t,x)$ satisfies
\begin{equation}\label{z1}
\p_t\nabla n_\delta-\Delta(\nabla n_\delta)=-\nabla(u\cdot\nabla n)+\nabla( |\nabla n|^2(n+\underline d)):=\nabla S_3, \quad \nabla n_\delta(t,x)|_{t=0}=0.
\end{equation}
Using Duhamel principle, for $k=0,1$, we obtain that
\begin{equation}\label{k0}
\begin{aligned}
\|\nabla\nd\|_{L^2}&\le C\int_0^t(1+t-\tau)^{-\frac54}\big(\|u\cdot \nabla n\|_{L^1}+\||\nabla n|^2(n+\underline d)\|_{L^1}+\|\nabla S_3\|_{L^2}\big)d\tau.
\end{aligned}
\end{equation}
By routine calculation, it is easy to see
\begin{equation}\label{k1}
\begin{aligned}
\|u\cdot\nabla n\|_{L^1}+\||\nabla n|^2(n+\underline d)\|_{L^1}&\le \|u\|_{L^2}\|\nabla n\|_{L^2}+\|\nabla n\|_{L^2}\|\nabla n\|_{L^3}\|n\|_{L^6}+\underline d\|\nabla n\|_{L^2}\|\nabla n\|_{L^2}\\&
\le C\big(\|u\|_{L^2}\|\nabla n\|_{L^2}+\|\nabla n\|_{L^2}\|\nabla n\|_{H^1}\|\nabla n\|_{L^2}+\|\nabla n\|_{L^2}^2\big),
\end{aligned}
\end{equation}
and
\begin{equation}\label{k2}
\begin{aligned}
\|\nabla S_3\|_{L^2}&\le \|\nabla u\|_{L^2}\|\nabla n\|_{L^\infty}+\|\nabla^2n\|_{L^2}\|u\|_{L^\infty}+2\|\nabla^2n\|_{L^2}\|\nabla n\|_{L^3}\|n\|_{L^6}\\&
\quad +\|\nabla n\|_{L^2}\|\nabla n\|_{L^3}\|\nabla n\|_{L^6}+\|\nabla ^2n\|_{L^3}\|\nabla n\|_{L^6}\\&
\le \|\nabla u\|_{L^2}\|\nabla^2n\|_{L^2}^\frac12\|\nabla^3n\|_{L^2}^\frac12+\|\nabla^2n\|_{L^2}\|\nabla u\|_{L^2}^\frac12\|\nabla^2u\|_{L^2}^\frac12\\&
\quad +\|\nabla n\|_{L^2}^\frac32\|\nabla^2n\|_{L^2}^\frac32+\|\nabla^2n\|_{L^2}^\frac32\|\nabla^3n\|_{L^2}^\frac12\\&
\le C\|\nabla(u,\nabla n)\|_{L^2}^\frac32\|\nabla^2(u,\nabla n)\|_{L^2}^\frac12+\|\nabla n\|_{L^2}^\frac32\|\nabla^2n\|_{L^2}^\frac32,
\end{aligned}
\end{equation}
where we used the Sobolev inequalities
\begin{equation}\label{m1}
\begin{aligned}
\|\nabla n\|_{L^\infty}+\|\nabla ^2n\|_{L^3}\le \|\nabla^2n\|_{L^2}^\frac12\|\nabla^3n\|_{L^2}^\frac12, \quad \|u\|_{L^\infty}\le \|\nabla u\|_{L^2}^
\frac12\|\nabla^2u\|_{L^2}^\frac12, \quad
 \|\nabla n\|_{L^3}\le \|\nabla n\|_{L^2}^\frac12\|\nabla^2n\|_{L^2}^\frac12.
\end{aligned}
\end{equation}
Adding \eqref{k1}, \eqref{k2} into \eqref{k0}, and using the decay estimate \eqref{decay}, we obtain that
\begin{equation*}
\begin{aligned}
\|\nabla \nd(t)\|_{L^2}&\le C\int_0^t(1+t-\tau)^{-\frac54}\big((1+\tau)^{-\frac32}+(1+\tau)^{-\frac94}+(1+\tau)^{-\frac98}\|\nabla^2(u,\nabla n)\|^\frac12_{L^2}\big)d\tau\\&
\le C(1+t)^{-\frac54}+C\int_0^t(1+t-\tau)^{-\frac54}(1+\tau)^{-\frac98}\|\nabla^2(u,\nabla n)\|^\frac12_{L^2}\big)d\tau.
\end{aligned}
\end{equation*}
Using the H\"{o}lder inequality, and the uniform estimate \eqref{uni}, we have
\begin{equation*}
\begin{aligned}
&\int_0^t(1+t-\tau)^{-\frac54}(1+\tau)^{-\frac98}\|\nabla^2(u,\nabla n)\|^\frac12_{L^2}\big)d\tau\\&
\le \big(\int_0^t(1+t-\tau)^{-\frac53}(1+\tau)^{-\frac32}d\tau\big)^\frac34\big(\int_0^t\|\nabla^2(u,\nabla n)\|_{L^2}^2d\tau\big)^\frac14\\&
\le C(1+t)^{-\frac98},
\end{aligned}
\end{equation*}
which implies that for all $t\ge 0$, there holds
\begin{equation}\label{kkk0}
\|\nabla\nd(t)\|_{L^2}\le C(1+t)^{-\frac98}.
\end{equation}

Applying $\nabla$ to \eqref{z1}, then multiplying the resulting equation by $\nabla^2\nd$ and integrating over $\mathbb R^3$, we get
\begin{equation}\label{k3}
\begin{aligned}
\frac12\frac{d}{dt}\int|\nabla^2\nd|^2dx+\int|\nabla^3\nd|^2dx=\int\nabla^2S_3\cdot\nabla^2\nd dx\le\frac12\|\nabla S_3\|_{L^2}^2+\frac12\|\nabla^3\nd\|^2_{L^2}.
\end{aligned}
\end{equation}
According to \eqref{k2} and the decay estimate \eqref{decay1}, for all $t\ge T_2$, there holds
\begin{equation*}
\|\nabla S_3(t)\|_{L^2}\le C(1+t)^{-\frac{11}{4}},
\end{equation*}
which together with \eqref{k3} gives rise to
\begin{equation}\label{k4}
\frac{d}{dt}\int|\nabla^2\nd|^2dx+\int|\nabla^3\nd|^2dx\le C(1+t)^{-\frac{11}{2}},\quad \text{for} \quad t\ge T_2.
\end{equation}
Define $S_0:=\{\xi\in \mathbb R^3||\xi|\le (\frac{R}{1+t})^\frac12\}$, then we can split the phase space $\mathbb R^3$ into two time-dependent regions. Here $R$ is a constant defined below. By straightforward calculation, we get
\begin{equation}\label{k5}
\begin{aligned}
\int |\nabla^3\nd|^2dx&\ge \int_{\mathbb R^3/S_0}|\xi|^4|\widehat{\nabla\nd}|^2d\xi\ge\frac{R}{1+t} \int_{\mathbb R^3}|\xi|^2|\widehat{\nabla\nd}|^2d\xi-\frac{R}{1+t} \int_{S_0}|\xi|^2|\widehat{\nabla\nd}|^2d\xi\\&
\ge \frac{R}{1+t} \int_{\mathbb R^3}|\xi|^2|\widehat{\nabla\nd}|^2d\xi-\frac{R^2}{(1+t)^2}\int_{S_0}|\widehat{\nabla\nd}|^2d\xi.
\end{aligned}
\end{equation}
Adding \eqref{k5} into \eqref{k4}, and using Plancherel equality, then for $t\ge T_2$, there holds
\begin{equation*}
\frac{d}{dt}\int|\nabla^2\nd|^2dx+\frac{R}{1+t}\|\nabla^2\nd\|_{L^2}^2
\le \frac{R^2}{(1+t)^2}\|\nabla\nd\|_{L^2}^2+C(1+t)^{-\frac{11}{2}}
\le CR^2(1+t)^{-\frac{17}{4}},
\end{equation*}
where we used \eqref{kkk0} in the last step.
Choosing $R=\frac{15}{4}$, and multiplying the resulting inequality by $(1+t)^{\frac{15}{4}}$, it follows that
\begin{equation*}
\frac{d}{dt}\big[(1+t)^{\frac{15}{4}}\|\nabla^2n_\delta(t)\|_{L^2}^2\big]\le C(1+t)^\frac12, \quad t\ge T_2.
\end{equation*}
The integration over $[T_2,t]$ and the uniform bound \eqref{uni} give rise to
\begin{equation}\label{m3}
\|\nabla^2\nd(t)\|_{L^2}^2\le C(1+t)^{-\frac{13}{4}},\quad t\ge T_2.
\end{equation}
This together with \eqref{l2}, we have
\begin{equation}\label{yyy1}
\|\nabla^2n\|_{L^2}
\ge \|\nabla^2\bar n\|_{L^2}-\|\nabla^2 \nd\|_{L^2}
\ge C(1+t)^{-\frac54}-C(1+t)^{-\frac{13}{8}}
\ge C(1+t)^{-\frac54},
\end{equation}
for large time $t$.

Next, we shall estimate the lower bound decay rate of $\nabla^3n$.
Applying $\nabla^2$ to \eqref{z1}, then multiplying the resulting equation by $\nabla^3\nd$ and integrating over $\mathbb R^3$, we obtain that
\begin{equation}\label{m2}
\frac12\int|\nabla^3\nd|^2dx+\int|\nabla^4\nd|^2dx=\int\nabla^3S_3\cdot\nabla^3\nd dx\le\|\nabla^2S_3\|_{L^2}\|\nabla^4\nd\|_{L^2}.
\end{equation}
Recalling the definition of $S_3$, routine calculation gives
\begin{equation*}
\begin{aligned}
\|\nabla^2S_3\|_{L^2}&\le \|\nabla^2u\|_{L^2}\|\nabla n\|_{L^\infty}+2\|\nabla u\|_{L^3}\|\nabla^2n\|_{L^6}+\|\nabla^3n\|_{L^2}\|u\|_{L^\infty}\\&
\quad +\|n\|_{L^\infty}\|\nabla^2n\|_{L^3}\|\nabla^2n\|_{L^6}
+\|\nabla^2n\|_{L^3}\|\nabla^2n\|_{L^6}
+\|n\|_{L^\infty}\|\nabla n\|_{L^\infty}\|\nabla^3n\|_{L^2}\\&
\quad +\|\nabla n\|_{L^\infty}\|\nabla^3n\|_{L^2}+\|\nabla n\|_{L^\infty}\|\nabla n\|_{L^3}\|\nabla^2n\|_{L^6}\\&
\le C\|\nabla(u,\nabla n)\|_{L^2}^\frac12\|\nabla^2(u,\nabla n)\|_{L^2}^\frac12\|\nabla^2(u,\nabla n)\|_{L^2}+\|\nabla n\|_{L^2}^\frac12\|\nabla^2n\|_{L^2}\|\nabla^3n\|_{L^2}^\frac32,
\end{aligned}
\end{equation*}
where we used \eqref{m1} and the interpolation inequality
$
\|n\|_{L^\infty}\le \|\nabla n\|_{L^2}^\frac12\|\nabla^2n\|_{L^2}^\frac12.
$
According to the decay estimate \eqref{decay1}, for all $t\ge T_2$, there holds
\begin{equation*}
\|\nabla^2S_3(t)\|_{L^2}\le C(1+t)^{-\frac{13}{4}},
\end{equation*}
which together with \eqref{m2} yields directly
\begin{equation*}
\frac{d}{dt}\int|\nabla^3\nd|^2dx+\int|\nabla^4\nd|^2dx\le C(1+t)^{-\frac{13}{2}}, \quad t\ge T_2.
\end{equation*}
Similarly, we have
\begin{equation*}
\|\nabla^4\nd\|_{L^2}^2\ge \frac{R}{1+t}\|\nabla^3\nd\|_{L^2}^2-\frac{R^2}{(1+t)^2}\|\nabla^2\nd\|_{L^2}^2.
\end{equation*}
Therefore, we have
\begin{equation}\label{m4}
\frac{d}{dt}\int|\nabla^3\nd|^2dx+\frac{R}{1+t}\|\nabla^3\nd\|_{L^2}^2 \le \frac{R^2}{(1+t)^2}\|\nabla^2\nd\|_{L^2}^2+C(1+t)^{-\frac{13}{2}}
\le R^2(1+t)^{-\frac{21}{4}},
\end{equation}
where we used \eqref{m3} in the last step.
Multiplying $(1+t)^{\frac{19}{4}}$ on the both side of \eqref{m4}, and choosing $R=\frac{19}{4}$, we get
\begin{equation*}
\frac{d}{dt}\big((1+t)^{\frac{19}{4}}\|\nabla^3\nd(t)\|_{L^2}^2\big)\le C(1+t)^{-\frac12},\quad t\ge T_2.
\end{equation*}
Integrating from $[T_2, t]$ and using the uniform estimate \eqref{uni}, we get
\begin{equation*}
\|\nabla^3\nd(t)\|_{L^2}^2\le C(1+t)^{-\frac{17}{4}},\quad t\ge T_2.
\end{equation*}
Then we have
\begin{equation}\label{yyy2}
\|\nabla^3n(t)\|_{L^2}
\ge\|\nabla^3\bar n(t)\|_{L^2}-\|\nabla^3\nd(t)\|_{L^2}
\ge C(1+t)^{-\frac74}-C(1+t)^{-\frac{17}{8}}
\ge C(1+t)^{-\frac74},
\end{equation}
for large time. The combination of \eqref{l1}, \eqref{yyy1} and \eqref{yyy2} completes the proof of this lemma.
\end{proof}

We are now concerned with the lower bound of decay rate for $(\r, u)$. Let us define $m:=\rho u$, we rewrite $\eqref{liqu}_1$ and $\eqref{liqu}_2$ in the perturbation form
\begin{equation*}
  	\left\{\begin{aligned}
  &\partial_t\r+\dive m=0,\\
   &\partial_tm-\mu\Delta m-(\mu+\lambda)\nabla\dive m+P'(1)\nabla\r=F,
  	\end{aligned}\right.
  \end{equation*}
where
\begin{equation*}
\begin{aligned}
F&:=-\mu\Delta(\r m)-(\mu+\lambda)\nabla\dive(\r u)-\dive((1+\r)u\otimes u)\\&
\quad -\nabla\big(P(1+\r)-P(1)-P'(1)\r\big)-\dive(\nabla d\odot\nabla d)-\frac12\nabla(|\nabla d|^2),
\end{aligned}
\end{equation*}
here $\nabla d\odot\nabla d=(\langle\p_id,\p_jd \rangle)_{1\le i,j\le 3}$.
It is easy to check
\begin{equation*}
\nabla d\cdot\Delta d=\dive (\nabla d\odot\nabla d)-\frac12\nabla(|\nabla d|^2).
\end{equation*}
In order to obtain the lower decay estimate, we need to consider the linearized system
\begin{equation}\label{n1}
  	\left\{\begin{aligned}
  &\partial_t\bar\r+\dive \bar m=0,\\
   &\partial_t\bar m-\mu\Delta\bar m-(\mu+\lambda)\nabla\dive \bar m+P'(1)\nabla\bar\r=0,
  	\end{aligned}\right.
  \end{equation}
  with the initial data
  \begin{equation*}
  (\bar\r,\bar m)|_{t=0}=(\r_0,m_0).
  \end{equation*}

  Next, we introduce the following lemma, which can be found in \cite{li-zhang}.
  \begin{lemm}
  If $|\widehat{\r_0}|\ge c_0$, $\widehat{m_0}=0$,  $0\le|\xi|\ll1$,
  with $c_0>0$ a constant, then the solution $(\bar\r,\bar m)$ of the linearized system \eqref{n1} has the following estimate
  \begin{equation*}
  \min\{\|\nabla^k\bar\r\|_{L^2},\|\nabla^k \bar m\|_{L^2}\}\ge C(1+t)^{-\frac34-\frac k2}, \quad k=0,1,2
  \end{equation*}
  for large time $t$.
  \end{lemm}

  Finally, we establish the lower bound of decay rate for the global solution of Cauchy problem \eqref{liqu1}.
  \begin{lemm}\label{lemm33}
    If $|\widehat{\r_0}|\ge c_0$, $\widehat{m_0}=0$,  $0\le|\xi|\ll1$,
  with $c_0>0$ a constant, then the solution $(\r,u)$ of the system \eqref{liqu1} has the following estimate
  \begin{equation*}
  \min\{\|\nabla^k\r(t)\|_{L^2},\|\nabla^k u(t)\|_{L^2}\}\ge C (1+t)^{-\frac34-\frac k2}, \quad k=0,1,2
  \end{equation*}
  for large time $t$.
  \end{lemm}
  \begin{proof}
  Define $\rd:=\r-\bar\r$, $\md:=m-\bar m$, then $(\rd,\md)$ satisfies
  \begin{equation*}
  	\left\{\begin{aligned}
  &\partial_t\rd+\dive\md=0,\\
   &\partial_t\md-\mu\Delta\md-(\mu+\lambda)\nabla\dive \md+P'(1)\nabla\rd=F,
  	\end{aligned}\right.
  \end{equation*}
  with the initial data
  \begin{equation*}
  (\rd,\md)|_{t=0}=(0,0).
  \end{equation*}

  First, we investigate the time decay rate for $L^2-$norm of the low-frequency part $\nabla^k(\rd^l,\md^l)$, $k=0,1$. By Duhamel principle, we have
  \begin{equation}\label{o1}
  \begin{aligned}
  \|\nabla^k(\rd^l,\md^l)\|_{L^2}&\le C\int_0^t(1+t-\tau)^{-\frac54-\frac k2}\Big[\|(1+\r)u\otimes u\|_{L^1}+\|P(1+\r)-P(1)-P'(1)\r\|_{L^1}\\&
  \qquad  \quad+\|\nabla d\odot\nabla d\|_{L^1}+\||\nabla d|^2\|_{L^1}\Big]d\tau
 +C\int_0^t(1+t)^{-\frac74-\frac k2}\|\r u\|_{L^1}d\tau.
  \end{aligned}
  \end{equation}
  It is easy to check that
\begin{equation*}
\begin{aligned}
&\|(1+\r)u\otimes u\|_{L^1}+\|P(1+\r)-P(1)-P'(1)\r\|_{L^1}
  +\|\nabla d\odot\nabla d\|_{L^1}+\||\nabla d|^2\|_{L^1}
  +\|\r u\|_{L^1}\\&
  \le C\big( \|\rho\|_{L^\infty}\|u\|_{L^2}^2+\|\r\|_{L^2}^2+\|\nabla n\|_{L^2}^2+\|\r\|_{L^2}\|u\|_{L^2}\big)\\&
  \le C\big(\|u\|_{L^2}^2+\|\r\|_{L^2}^2+\|\nabla n\|_{L^2}^2+\|\r\|_{L^2}\|u\|_{L^2}\big),
\end{aligned}
\end{equation*}
this together with \eqref{o1} gives for $k=0,1$
\begin{equation*}
\begin{aligned}
\|\nabla^k(\rd^l,\md^l)\|_{L^2}&\le C\int_0^t(1+t-\tau)^{-\frac54-\frac k2}(1+\tau)^{-\frac32}d\tau +C\int_0^t(1+t-\tau)^{-\frac74-\frac k2}(1+\tau)^{-\frac32}d\tau\\&
\le C(1+t)^{-\frac54-\frac k4}.
\end{aligned}
\end{equation*}
Recalling the fact that
\begin{equation*}
\r^l=\bar\r^l+\rd^l, \quad m^l=\bar m^l+\md^l,
\end{equation*}
and hence we have for $k=0,1$
\begin{equation}\label{o3}
\|\nabla^k(\r(t)\|_{L^2}
\ge C( \|\nabla^k\bar\r^l(t)\|_{L^2}-\|\nabla^k\rd^l(t)\|_{L^2})
\ge C(1+t)^{-\frac34-\frac k2}-C(1+t)^{-\frac54-\frac k4}
\ge C(1+t)^{-\frac34-\frac k2},
\end{equation}
for large time $t$.
Using \eqref{lbrho}, it is easy to see that
\begin{equation}\label{o4}
\|u(t)\|_{L^2}=\|\frac m\rho(t)\|_{L^2}\ge \underline \rho\|m(t)\|_{L^2}\ge C\|m^l(t)\|_{L^2}\ge C(1+t)^{-\frac34}.
\end{equation}
Because $u=m-\frac{\r m}{\rho}$, we have
\begin{equation}\label{o2}
\|\nabla u^l(t)\|_{L^2}\ge \|\nabla m^l(t)\|_{L^2}-\|\nabla(\frac{\r m}{\rho})^l(t)\|_{L^2}
\end{equation}
According to Bernstein inequality \eqref{bern} and Hausdorff-Young inequality \eqref{hy}, we have
\begin{equation*}
\begin{aligned}
\|\nabla(\frac{\r m}{\rho})^l\|_{L^2}
&\le C\|(\frac{\r m}{\rho})^l\|_{L^2}
= C\|\mathcal F^{-1}\big(\chi_0(\xi)\widehat{(\frac{\r m}{\rho})}\big)\|_{L^2}
=C\|\chi_0(\xi)\widehat{(\frac{\r m}{\rho})}\|_{L^2}
\le C\|\chi_0(\xi)\|_{L^2}\|\widehat{(\frac{\r m}{\rho})}\|_{L^\infty}\\
&\le C\|\frac{\r m}{\rho}\|_{L^1}
\le C\|\r\|_{L^2}\|m\|_{L^2}
\le C\|\r\|_{L^2}\|u\|_{L^2}\le C(1+t)^{-\frac32},
\end{aligned}
\end{equation*}
this together with \eqref{o2} yields
\begin{equation}\label{o5}
\|\nabla u(t)\|_{L^2}\ge \|\nabla u^l(t)\|_{L^2}\ge C(1+t)^{-\frac54}-C(1+t)^{-\frac32}\ge C(1+t)^{-\frac54}
\end{equation}
for large $t$.
The combination of \eqref{o3}, \eqref{o4} and \eqref{o5} 
gives that
\begin{equation}\label{rr78}
  \min\{\|\nabla^k\r(t)\|_{L^2},\|\nabla^k u(t)\|_{L^2}\}\ge C (1+t)^{-\frac34-\frac k2}, \quad k=0,1
  \end{equation}
  for large $t$. Next, we apply interpolation inequality to obtain the lower bound of decay rate for the second order spatial derivative of large solution (cf. \cite{wu-zhang}). The decay estimate \eqref{rr78} together with interpolation inequality:
  \begin{equation*}
  \|\nabla f\|_{L^2}\le \|f\|_{L^2}^{\frac12}\|\nabla^2f\|_{L^2}^{\frac12}
  \end{equation*}
  gives rise to 
  \begin{equation}\label{rr79}
 \min\{\|\nabla^2\r(t)\|_{L^2},\|\nabla^2u(t)\|_{L^2}\}\ge C(1+t)^{-\frac74}
  \end{equation}
 for large time $t$. The combination of \eqref{rr78} and \eqref{rr79} completes the proof of this lemma.
    \end{proof}

   \textbf{Proof of Theorem \ref{th3}:} With the combination of Lemmas \ref{lemm31} and \ref{lemm33}, we can complete the proof of Theorem \ref{th3}.

\begin{appendices}
\section{Estimates on the linearized equations}\label{appendixA}

\quad
We adopt the following notations
\begin{equation*}
\Lambda:=(-\Delta)^{-\frac12},\quad b:=\Lambda^{-1}\dive u.
\end{equation*}
Then $u=-\Lambda^{-1}\nabla b-\Lambda^{-1}\dive (\Lambda^{-1}{\rm curl} u)$, we get from \eqref{liqu1} that
\begin{equation}\label{liqu2}
  	\left\{\begin{aligned}
	&\p_t\r+\L b=S_1,\\
	&\p_t b-(2\mu+\lambda)\Delta b-P'(1)\L \r=\L^{-1}\dive S_2,\\
	&\p_t n-\Delta n=S_3,
    	\end{aligned}\right.
  \end{equation}
  While $\mathcal Pu=\L^{-1}{\rm curl} u$ satisfies
  \begin{equation}\label{liqu4}
  	\left\{\begin{aligned}
	&\p_t\mathcal Pu-\mu\Delta\mathcal Pu=\mathcal PS_2,\\
	&\mathcal Pu(t,x)|_{t=0}=\mathcal Pu_0(x).
    	\end{aligned}\right.
  \end{equation}
  In fact, to derive the estimate of $u$, we only need to estimate $b$ and $\mathcal Pu$.
  Taking the Fourier transform of the first and second equations in \eqref{liqu2}, applying $\nabla$ operator to the third equation in \eqref{liqu2} then taking the Fourier transform, then we get
\begin{equation}\label{liquu}
  	\left\{\begin{aligned}
	&\p_t\hat\r+|\xi|\hat b=\hat{S}_1,\\
	&\p_t \hat b+(2\mu+\lambda)|\xi|^2\hat b-P'(1)|\xi|\hat\r=\widehat{\L^{-1}\dive S_2},\\
	&\p_t \widehat{\nabla n}+|\xi|^2\widehat{\nabla n}=\widehat{\nabla S_3}.
    	\end{aligned}\right.
  \end{equation}
  In other words, it is
\begin{equation*}
\frac{d}{dt}\hat U+G(|\xi|)\hat U=(\hat{S}_1,\widehat{\L^{-1}\dive S_2},\widehat{\nabla S_3})^T,
\end{equation*}
 where
 \begin{equation*}
 \hat U=(\hat\r,\hat b,\widehat {\nabla n})^T,
 \end{equation*}
  and
  \begin{gather}\label{gg1}
G(|\xi|)=
\begin{pmatrix}
0            & |\xi|  & 0 \\
-P'(1)|\xi| & (2\mu+\lambda)|\xi|^2 & 0\\
0 &  0 & |\xi|^2
\end{pmatrix}
\end{gather}
Thus the linearized system of \eqref{liquu} could be rewritten as:
\begin{equation}\label{liqu3}
  	\left\{\begin{aligned}
	&\p_t\hat\r+|\xi|\hat b=0,\\
	&\p_t \hat b+(2\mu+\lambda)|\xi|^2\hat b-P'(1)|\xi|\hat\r=0,\\
	&\p_t \widehat{\nabla n}+|\xi|^2\widehat{\nabla n}=0.
    	\end{aligned}\right.
  \end{equation}

  Next, we show the estimates of the low-frequency part for the solution of the linearized system \eqref{liqu3}.

  \subsection{Low-frequency analysis}
From \eqref{liqu3}, we easily obtain
  \begin{equation}\label{f1}
  \frac12\frac{d}{dt}\big(P'(1)|\hat \r|^2+|\hat b|^2+|\widehat {\nabla n}|^2\big)+(2\mu+\lambda)|\xi|^2|\hat b|^2+|\xi|^2|\widehat {\nabla n}|^2=0.
  \end{equation}
  Multiplying $\eqref{liqu3}_1$ and $\overline{\eqref{liqu3}_2}$ by $\bar{\hat b}$ and $\hat\r$, respectively, and summing up the resulting equations, we get
  \begin{equation}\label{f2}
  \frac{d}{dt}Re(\hat\r\bar{\hat b})-P'(1)|\xi||\hat\r|^2+|\xi||\hat b|^2=-(2\mu+\lambda)|\xi|^2Re(\hat\r\bar{\hat b}).
  \end{equation}
  Adding \eqref{f1} with $-\delta_1|\xi|\times\eqref{f2}$ gives rise to
  \begin{equation}\label{f3}
  \begin{aligned}
 & \frac12\frac{d}{dt}\big(P'(1)|\hat\r|^2+|\hat b|^2+|\widehat{\nabla n}|^2-2\delta_1|\xi|Re(\hat\r\bar{\hat b})\big)+(2\mu+\lambda)|\xi|^2|\hat b|^2\\&
  \quad +|\xi|^2|\widehat {\nabla n}|^2+\delta_1P'(1)|\xi|^2|\r|^2-\delta_1|\xi|^2|\hat b|^2\\&
  =\delta_1(2\mu+\lambda)|\xi|^3Re(\hat\r\bar{\hat b})\\&
  \le \frac{\delta_1}{2P'(1)}(2\mu+\lambda)^2|\xi|^4|\hat b|^2+\frac{\delta_1P'(1)}{2}|\xi|^2|\hat\r|^2.
  \end{aligned}
  \end{equation}
  we choose the constant $\delta_1$ is a small fixed constant which satisfies
  \begin{equation*}
  \delta_1\le \min\{\frac12,\frac{2\mu+\lambda}{4}, \frac{\sqrt{P'(1)}}{2}\}.
  \end{equation*}
  Then \eqref{f3} implies that
   \begin{equation}\label{f4}
  \begin{aligned}
 & \frac12\frac{d}{dt}\big(P'(1)|\hat\r|^2+|\hat b|^2+|\widehat{\nabla n}|^2-2\delta_1|\xi|Re(\hat\r\bar{\hat b})\big)+\frac{3(2\mu+\lambda)}{4}|\xi|^2|\hat b|^2\\&
  \quad +|\xi|^2|\widehat {\nabla n}|^2+\frac{\delta_1P'(1)}{2}|\xi|^2|\r|^2\\&
    \le \frac{\delta_1}{2P'(1)}(2\mu+\lambda)^2|\xi|^4|\hat b|^2.
  \end{aligned}
  \end{equation}
  Recalling that $r_0$ is a small constant in \eqref{yy1}, then for $|\xi|\le r_0\le \min\{\frac12,\sqrt{\frac{P'(1)}{2\mu+\lambda}}, \frac{\sqrt{P'(1)}}{2}\}$, it follows from \eqref{f4} that
   \begin{equation}\label{f5}
  \begin{aligned}
 & \frac12\frac{d}{dt}\big(P'(1)|\hat\r|^2+|\hat b|^2+|\widehat{\nabla n}|^2-2\delta_1|\xi|Re(\hat\r\bar{\hat b})\big)+\frac{2\mu+\lambda}{2}|\xi|^2|\hat b|^2\\&
  \quad +|\xi|^2|\widehat {\nabla n}|^2+\frac{\delta_1P'(1)}{2}|\xi|^2|\r|^2\le 0.
  \end{aligned}
  \end{equation}
  We denote
  \begin{equation*}
  \mathcal L_l(t,\xi):=\frac12\big(P'(1)|\hat\r|^2+|\hat b|^2+|\widehat{\nabla n}|^2-2\delta_1|\xi|Re(\hat\r\bar{\hat b})\big).
  \end{equation*}
  Since $\delta_1r_0\le\min\{\frac12, \frac{P'(1)}{2}\}$, then we get
  \begin{equation*}
    \mathcal L_l(t,\xi)\sim|\hat\r|^2+|\hat b|^2+|\widehat{\nabla n}|^2.
  \end{equation*}
  Then there exists a positive constant $C_5$ such that for any $|\xi|\le r_0$, there holds
  \begin{equation}\label{f6}
  C_5|\xi|^2\mathcal L _l(t,\xi)\le \frac{2\mu+\lambda}{2}|\xi|^2|\hat b|^2 +|\xi|^2|\widehat {\nabla n}|^2+\frac{\delta_1P'(1)}{2}|\xi|^2|\r|^2.
  \end{equation}
  The combination of \eqref{f5} and \eqref{f6} yields
  \begin{equation}\label{g1}
  \mathcal L_l(t,\xi)\le Ce^{-C_5|\xi|^2t}\mathcal L_l(0,\xi), \quad \text{for} \quad |\xi|\le r_0.
  \end{equation}

  Next, we establish the estimates of the medium-frequency   part for the solution of the linearized system \eqref{liqu3}.
  \subsection{Medium-frequency analysis}

  First, we shall check the characteristic polynomial of the matrix $G$ defined in \eqref{gg1}.
  \begin{equation*}
  \begin{aligned}
  |G(|\xi|)-\beta\mathbb I|&=(|\xi|^2-
  \beta)(\beta^2-(2\mu+\lambda)|\xi|^2\beta+P'(1)|\xi|^2)\\&
  :=-\beta^3+a_1(|\xi|)\beta^2-a_2(|\xi|)\beta+a_3(|\xi|),
\end{aligned}
  \end{equation*}
where
\begin{equation*}
\begin{aligned}
a_1(|\xi|)&=(2\mu+\lambda)|\xi|^2+|\xi|^2,\\
a_2(|\xi|)&=P'(1)|\xi|^2+(2\mu+\lambda)|\xi|^4,\\
a_3(|\xi|)&=P'(1)|\xi|^4.
\end{aligned}
\end{equation*}
According to Lienard-Chipart criterion \cite{lie-chi}, the roots of the function $ |G(|\xi|)-\beta\mathbb I|$ have positive real part if and only if the following inequalities hold
\begin{equation*}
a_1(|\xi|)>0, \quad \text{and} \quad a_1(|\xi|)a_2(|\xi|)-a_3(|\xi|)>0.
\end{equation*}
It is easy to check that
\begin{equation*}
\begin{aligned}
a_1(\xi|)&=(2\mu+\lambda)|\xi|^2+|\xi|^2>0,\\
a_1(|\xi|)a_2(|\xi|)-a_3(|\xi|)&=(2\mu+\lambda)P'(1)|\xi|^4+(2\mu+\lambda)^2|\xi|^6+(2\mu+\lambda)|\xi|^6>0.
\end{aligned}
\end{equation*}
Thus, following the discussion in Section 3.3 of \cite{danchin2}, we cloud claim that this fact implies the following lemma.
\begin{lemm}
For any given constants $r$ and $R$ with $0<r<R$, there exists a positive constant $\kappa$(depending only on $r$, $R$, $\mu$ and $\lambda$), such that
\begin{equation}\label{g4}
|e^{-tG(|\xi|)}|\le Ce^{-\kappa t}, \quad \text{for all}\quad  r\le |\xi|\le R\quad \text{and}\quad t\in\mathbb R^+.
\end{equation}
For the system \eqref{liqu3}, the inequality \eqref{g4} yields
\begin{equation}\label{g2}
\begin{aligned}
|(\hat\r,\hat b,\widehat{\nabla n})(t,\xi)|&=|e^{-tG(|\xi|)}(\hat\r,\hat b,\widehat{\nabla n})(0,\xi)|\\&
\le Ce^{-\kappa t}|(\hat\r,\hat b,\widehat{\nabla n})|, \quad \text{for all} \quad |\xi|\in[r,R],
\end{aligned}
\end{equation}
where $r$ and $R$ are any given positive constants.
\end{lemm}

Finally, we give the estimates of the linearized system \eqref{liqu4}.
\subsection{Estimates on $\widehat {\mathcal P u}(t,\xi)$}

Now we give the estimate of $\widehat{\mathcal Pu}$. The linearized equations of \eqref{liqu4} in Fourier variables has the following form
\begin{equation*}
\p_t\widehat{\mathcal Pu}+\mu|\xi|^2\widehat{\mathcal Pu}=0.
\end{equation*}
Direct calculations gives that for all $|\xi|\ge0$
\begin{equation}\label{g3}
|\widehat{\mathcal Pu}(t,\xi)|^2\le Ce^{-\mu|\xi|^2t}|\widehat{\mathcal Pu}(0,\xi)|^2.
\end{equation}

\section{The frequency decomposition}\label{appendixB}

\quad
Based on the Fourier transform, we can define a frequency decomposition $(f^l(x),f^m(x),f^h(x))$ for a function $f(x)\in L^2(\mathbb R^3)$ as follows
\begin{equation}\label{j1}
f^l(x)=\chi_0(D_x)f(x), \quad f^m(x)+(I-\chi_0(D_x)-\chi_1(D_x))f(x), \quad f^h(x)=\chi_1(D_x)f(x),
\end{equation}
where $\chi_0(D_x)$ and $\chi_1(D_x)$, $D_x=\frac{1}{\sqrt{-1}}\nabla=\frac{1}{\sqrt{-1}}(\p_{x_1},\p_{x_2},\p_{x_3})$, are the pseudo-differential operators with symbols $\chi_0(\xi)$ and $\chi_1(\xi)$, respectively. Here, $\chi_0(\xi)$ and $\chi_1(\xi)$ are two smooth cut-off functions satisfying $0\le \chi_0(\xi), \chi_1(\xi)\le 1, (\xi\in\mathbb R^3)$ and
\begin{equation*}
\chi_0(\xi)=\left\{
\begin{aligned}
1, \quad |\xi|<\frac{r_0}{2},\\
0,\quad |\xi|>r_0,
\end{aligned}
\right.\quad
\chi_1(\xi)=\left\{
\begin{aligned}
0, &\quad |\xi|<R_0,\\
1,&\quad |\xi|>R_0+1,
\end{aligned}
\right.
\end{equation*}
for some fixed constants $r_0$ and $R_0$ satisfying
\begin{equation}\label{yy1}
0<r_0\le\min\{\frac12,\sqrt{\frac{P'(1)}{2\mu+\lambda}},\frac{\sqrt{2P'(1)}}{2}\},
\end{equation}
and
\begin{equation}\label{yy2}
R_0\ge \max\{\sqrt{\frac{6C_2}{\mu}},\sqrt{\frac{3}{\mu}},\sqrt{\frac{6C_3}{\mu}},\sqrt{4C_2},\sqrt{4C_3}\}.
\end{equation}
Therefore, we have
\begin{equation*}
f(x)=f^l(x)+f^m(x)+f^h(x)
:=f^L(x)+f^h(x)
:=f^l(x)+f^H(x),
\end{equation*}
where we denote
\begin{equation}\label{decom}
f^L(x)=f^l(x)+f^m(x) \quad \text{and} \quad f^H(x)=f^m(x)+f^h(x).
\end{equation}
According to the definition \eqref{j1} and using the Plancherel theorem, we have the following inequalities.
\begin{lemm}\label{A1}
For $f(x)\in H^m(\mathbb R^3)$ and any given integers $k$, $k_0$ and $k_1$ with $k_0\le k\le k_1\le m$, it holds that
\begin{equation*}
\|\nabla^kf^l\|_{L^2}\le r_0^{k-k_0}\|\nabla^{k_0}f^l\|_{L^2},\quad \|\nabla^kf^l\|_{L^2}\le \|\nabla^{k_1}f\|_{L^2},
\end{equation*}
\begin{equation*}
\|\nabla^kf^h\|_{L^2}\le\frac{1}{R_0^{k_1-k}}\|\nabla^{k_1}f^h\|_{L^2}, \quad \|\nabla^kf^h\|_{L^2}\le \|\nabla^{k_1}f\|_{L^2},
\end{equation*}
and
\begin{equation*}
r_0^k\|f^m\|_{L^2}\le \|\nabla^kf^m\|_{L^2}\le R_0^k\|f^m\|_{L^2}.
\end{equation*}
\end{lemm}

The following Hausdorff-Young inequality is useful in this paper. The proof can be found in \cite{loukas}(see Proposition 2.2.16).
\begin{lemm}
When $f\in L^p(\mathbb R^3)$, $1\le p\le 2$, then $\hat f\in L^{p'}(\mathbb R^n)$, and there holds
\begin{equation}\label{hy}
\|\hat f\|_{L^{p'}}\le C\|f\|_{L^p},
\end{equation}
where $1/p+1/p'=1$.
\end{lemm}
At last, we introduce Bernstein inequality as follows
\begin{lemm}\label{bern}
Let $k$ be in $\mathbb N$. Let $(R_1,R_2)$ satisfy $0<R_1<R_2$. There exists a constant $C$ depending only on $R_1$, $R_2$, $k$, such that for all $1\le a\le b\le \infty$ and $u\in L^a$, we have
\begin{equation*}
\text{Supp}  \hat u\subset B(0,R_1\eta)\Rightarrow \sup_{|\alpha|=k}\|\p^\alpha u\|_{L^b}\le C^{k+1}\eta^{k+N(\frac1a-\frac1b)}\|u\|_{L^a}.
\end{equation*}
\end{lemm}
\end{appendices}

\section*{Acknowledgements}

Jincheng Gao's research was partially supported by NNSF of China (11801586),
and Natural Science Foundation of Guangdong Province of China (2020A1515110942).
Zheng-an Yao's research was partially supported by NNSF of China (11971496, 12026244).

\phantomsection


\addcontentsline{toc}{section}{\refname}
\begin{thebibliography}{99}
\bibitem{chen-huang}
Y. Chen, J. Huang, H. Xu, Z. Yao,
Global stability of large solutions to the 3-D compressible flow of liquid crystals,
Commun. Math. Sci. 18(4) (2020) 887-908.



\bibitem{danchin2}
 R. Danchin, B. Ducomet,
On a simplified model for radiating flows,
J. Evol. Equ. 14 (2014) 155-195.

\bibitem{ding2}
S. Ding, C. Wang, H. Wen,
Weak solution to compressible hydrodynamic flow of liquid crystals in dimension one,
Discrete Contin. Dyn. Syst. 15 (2011) 357-371.

 \bibitem{ding1}
S. Ding, J. Lin, C. Wang, H. Wen,
Compressible hydrodynamic flow of liquid crystals in 1-D,
Discrete Contin. Dyn. Syst. 32 (2012) 539-563.

\bibitem{du-wang}
Y. Du, K. Wang, Space-time regularity of the Koch and Tataru solutions to the liquid crystal equations,
SIAM J. Math. Anal. 45(6) (2013) 3838-3853.

\bibitem{duan1}
R. Duan, H. Liu, S. Ukai, T. Yang, Optimal $L^p$-$L^q$ convergence rate for the compressible Navier-Stokes equations with potential force, J. Differential Equations 238 (2007) 220-223.

\bibitem{duan2}
R. Duan, S. Ukai, T. Yang, H. Zhao, Optimal convergence rate for compressible Navier-Stokes equations with potential force, Math. Models Methods Appl. Sci. 17 (2007) 737-758.

\bibitem{erick}
J. Ericksen,
Hydrostatic theory of liquid crystals,
 Arch. Ration. Mech. Anal. 9 (1962) 371-378.

\bibitem{gao-lyu}
J.Gao, Z.Lyu, Z.Yao,
Lower Bound and Space-time Decay Rates of Higher Order Derivatives of Solution for the Compressible Navier-Stokes and Hall-MHD Equations,
arXiv: Analysis of PDEs, 2019.

 \bibitem{gao-tao}
 J. Gao, Q. Tao, Z. Yao,
 Strong solutions to the density-dependent incompressible nematic liquid crystal flows,
J. Differential Equations, 260 (2016) 3691-3748.


 \bibitem{gao-tao-yao}
J. Gao, Q. Tao, Z. Yao,
Long-time behavior of solution for the compressible nematic liquid crystal flows in $\mathbb R^3$,
J. Differential Equations 261 (2016) 2334-2383.

\bibitem{gao-wei-yao}
J. Gao, Z. Wei, Z. Yao,
 The optimal decay rate of strong solution for the compressible Navier-Stokes equations with large initial data,
 Phys. D 406 (2020), 132506, 9 pp.


\bibitem{gong-huang}
H. Gong, J. Huang, L. Liu, X. Liu,
 Global strong solutions of the 2D simplified Ericksen-Leslie system, Nonlinearity, 28(10) (2015) 3677-3694.
 
 \bibitem{gong-huang-li}
 H. Gong, T. Huang, J. Li,
 Nonuniqueness of nematic liquid crystal flows in dimension three,
 J. Differential Equations 263(2) (2017) 8630-8648. 
 
 \bibitem{gong-li-xu}
 H. Gong, J. Li, C. Xu,
 Local well-posedness of strong solutions to density-dependent liquid crystal system,
  Nonlinear Anal. 147 (2016) 26-44. 

 \bibitem{loukas}
L. Grafakos,
Classical Fourier Analysis Volume 249,
10.1007/978-1-4939-1194-3(2014).

\bibitem{guo-wang}
Y. Guo, Y. Wang,
Decay of dissipative equations and negative Sobolev spaces,
Comm. Partial Differential Equations 37 (2012) 2165-2208.

\bibitem{he-huang-wang}
L. He, J. Huang, C. Wang,
Global stability of large solutions to the 3D compressible Navier-Stokes equations,
Arch. Ration. Mech. Anal. 234(3) (2019) 1167-1222.


\bibitem{hin-wang}
J. Hineman, Y. Wang,
Well-posedness of nematic liquid crystal flow in $L^3_{uloc}(\mathbb R^3)$, Arch. Ration. Mech. Anal. 210(1) (2013) 177-218.

\bibitem{hoff-zumbrun}
D. Hoff, K. Zumbrun,
Multidimensional diffusion waves for the Navier-Stokes equations of compressible flow,
Indiana Univ. Math. J. 44 (1995) 604-676.

\bibitem{hong}
M. Hong,
 Global existence of solutions of the simplified Ericksen-Leslie system in dimension two,
Calc. Var. Partial Differ. Equ. 40(1-2) (2011) 15-36.

\bibitem{hong-xin}
M. Hong, Z. Xin,
Global existence of solutions of the liquid crystal flow for the Oseen-Frankmodel in $\mathbb R^2$,
Adv. Math. 231(3-4) (2012) 1364-1400.



\bibitem{hu-wu}
X. Hu, G. Wu,
Global existence and optimal decay rates for three-dimensional compressible viscoelastic flows,
SIAM J. Math. Anal. 45(5) (2013) 2815-2833.


 \bibitem{huang-ding}
 J. Huang, S. Ding,
 Compressible hydrodynamic flow of nematic liquid crystals with vacuum,
 J. Differential Equations 258(5) (2015) 1653-1684.


\bibitem{huang-wang-wen1}
T. Huang, C. Wang, H. Wen,
Strong solutions of the compressible nematic liquid crystal,
J. Differential Equations 252 (2012) 2222-2265.

\bibitem{huang-wang-wen2}
T. Huang, C. Wang, H. Wen,
Blow up criterion for compressible nematic liquid
crystal flows in dimension three,
Arch. Ration. Mech. Anal. 204 (2012) 285-311.


\bibitem{jiang-jiang-wang}
F. Jiang, S. Jiang, D. Wang,
On multi-dimensional compressible flows of nematic liquid crystals with large initial energy in a bounded domain,
J. Funct. Anal. 265 (2013) 3369-3397.

\bibitem{ka-ko}
Y. Kagei, T. Kobayashi,
On large time behavior of solutions to the compressible Navier-Stokes equations in the half space in $\mathbb R^3$,
Arch. Rational Mech. Anal. 165 (2002) 89-159.

\bibitem{lei}
Z. Lei, D. Li, X. Zhang,
Remarks of global well-posedness of liquid crystal flows and heat flows of
harmonic maps in two dimensions,
 Proc. Am. Math. Soc. 142(11) (2014) 3801-3810.


 \bibitem{les}
 F. Leslie,
 Some constitutive equations for liquid crystals,
 Arch. Ration. Mech. Anal. 28 (1968) 265-283.

 \bibitem{li-mat-zhang}
 H. Li, A. Matsumura, G. Zhang,
Optimal decay rate of the compressible Navier-Stokes-Poisson system in $\mathbb R^3$,
Arch. Ration. Mech. Anal. 196(2) (2010) 681-713.

 \bibitem{li-zhang}
H. Li, T. Zhang,
Large time behavior of isentropic compressible Navier-Stokes system in $\mathbb R^3$,
Math. Methods Appl. Sci. 34(6) (2011) 670-682.

\bibitem{li-jin}
J, Li,
 Global strong and weak solutions to inhomogeneous nematic liquid crystal flow in two dimensions, Nonlinear Anal. 99 (2014) 80-94.
 
 
\bibitem{lijinkai}
J. Li,
Global strong solutions to the inhomogeneous incompressible nematic liquid crystal flow, 
Methods Appl. Anal. 22(2) (2015) 201-220.

\bibitem{li-ti-xin}
 J. Li, E. Titi, Z. Xin, 
 On the uniqueness of weak solutions to the Ericksen-Leslie liquid crystal model in $\mathbb R^2$, 
 Math. Models Methods Appl. Sci. 26(4) (2016) 803-822.

\bibitem{li-xu-zhang}
J. Li, Z. Xu, J. Zhang,
Global existence of classical solutions with large oscillations and vacuum to the three-dimensional compressible nematic liquid crystal flows, J. Math. Fluid Mech. 20 (2018) 2105-2145.


 \bibitem{li-wang}
 X. Li, D. Wang,
 Global solution to the incompressible flow of liquid crystals,
 J. Differential Equations 252 (2012) 745-767.

 \bibitem{lie-chi}
A. Li\'enard, H. Chipart,
Sur le signe de la partie r\'elle des racines d'une \'equation alg\'ebrique,
J. Math. Pures Appl. 10 (1914) 291-346.





 \bibitem{lin-lin-wang}
 F. Lin, J. Lin, C. Wang,
 Liquid crystal flow in two dimensions,
 Arch. Ration. Mech. Anal. 197 (2010) 297-336.
 
 \bibitem{linf-wangc}
 F. Lin, C. Wang,
 On the uniqueness of heat flow of harmonic maps and hydrodynamic flow of nematic liquid crystals,
  Chin. Ann. Math. Ser. B 31(6) (2010) 921-938.
 

  \bibitem{lin-wang}
 F. Lin, C. Wang,
 Global existence of weak solutions of the nematic liquid crystal flow in dimension three,
 Comm. Pure Appl. Math. 69 (8) (2016) 1532-1571.

 \bibitem{lin-lai-wang}
 J. Lin, B. Lai, C. Wang,
 Global finite energy weak solutions to the compressible nematic liquid crystal flow in dimension three,
 SIAM J. Math. Anal. 47 (2015) 2952-2983.


 \bibitem{liu-wang}
T. Liu, W. Wang,
The point wise estimates of diffusion waves for the Navier-Stokes equations in odd multi-dimensions,
Comm. Math. Phys. 196 (1998) 145-173.

 \bibitem{ma}
S. Ma,
Classical solutions for the compressible liquid crystal flows with nonnegative initial densities,
J. Math. Anal. Appl. 397 (2013) 595-618.

\bibitem{ma-gong-li}
W. Ma, H. Gong, J. Li,
Global strong solutions to incompressible Ericksen-Leslie system in $\mathbb R^3$,
 Nonlinear Anal. 109 (2014) 230-235.

 \bibitem{mat-ni}
A. Matsumura, T. Nishida,
The initial value problem for the equations of motion of viscous and heat-conductive gases,
J. Math. Kyoto Univ. 20(1) (1980) 67-104.

\bibitem{ponce}
G. Ponce,
Global existence of small solution to a class of nonlinear evolution equations, Nonlinear Anal. 9 (1985) 339-418.

\bibitem{tan-wang}
Z. Tan, H. Wang,
Global existence and optimal decay rate for the strong solutions in $H^2$ to the 3-D compressible Navier-Stokes equations without heat conductivity,
J. Math. Anal. Appl. 394(2) (2012) 571-580.

\bibitem{ukai}
S. Ukai, T. Yang, H. Zhao, Convergence rate for the compressible Navier-Stokes equations with external force, J. Hyperbolic Differ. Equ. 3(3) (2006) 561-574.


 \bibitem{wang}
 C. Wang,
 Well-posedness for the heat flow of harmonic maps and the liquid crystal flow with rough initial data,
 Arch. Ration. Mech. Anal. 200 (2011) 1-19.


\bibitem{wang-wen}
W. Wang, H. Wen,
Global well-posedness and time-decay estimates for compressible Navier-Stokes equations with reaction diffusion,
Sci China Math, 64 (2021).


\bibitem{wang-tan}
Y. Wang, Z. Tan,
Global existence and optimal decay rate for the strong solutions in $H^2$ to the compressible Navier-Stokes equations,
Appl. Math. Lett. 24(11) (2011) 1778-1784.


   \bibitem{ding-wen}
 H. Wen, S. Ding,
 Solution of incompressible hydrodynamic flow of liquid crystals,
  Nonlinear Anal. Real World Appl. 12 (2011) 1510-1531.
  
  \bibitem{wu-zhang}
G. Wu, Y. Zhang, L. Zhou, 
Optimal large-time behavior of the two-phase fluid model in the whole space, SIAM J. Math. Anal. 52(6) (2020) 5748-5774. 



\bibitem{xu-zhang}
X. Xu, Z. Zhang,
Global regularity and uniqueness of weak solution for the 2-D liquid crystal flows,
J. Differential Equations 252(2) (2012) 1169-1181.



\bibitem{zhang-li-zhu}
G. Zhang, H. Li, C. Zhu,
Optimal decay rate of the non-isentropic compressible Navier-Stokes-Poisson system in  $\mathbb R^3$,
J. Differential Equations 250(2) (2011) 866-891.

\end{thebibliography}
\end{document}